 \documentclass[11pt]{article}
 \UseRawInputEncoding
\usepackage{amssymb, amsthm, amsmath, amscd}
\usepackage{tikz-cd}
\newtheorem{thm}{Theorem}[section]
\newtheorem{lem}[thm]{Lemma}
\newtheorem{prop}[thm]{Proposition}
\newtheorem{cor}[thm]{Corollary}
\newtheorem{NN}[thm]{}
\theoremstyle{definition}\newtheorem{df}[thm]{Definition}
\theoremstyle{definition}
\theoremstyle{definition}
\newtheorem{nota}[thm]{Notation}

\renewcommand{\phi}{\varphi}

\newcommand{\red}{\textcolor{red}}

\newcommand{\N}{\mathbb{N}}

\newcommand{\R}{\mathbb{R}}
\newcommand{\C}{\mathbb{C}}

\newcommand{\Aff}{\operatorname{Aff}}

\newcommand{\dt}{\delta}
\newcommand{\ep}{\epsilon}
\newcommand{\la}{\langle}
\newcommand{\ra}{\rangle}
\newcommand{\andeqn}{\,\,\,{\rm and}\,\,\,}
\newcommand{\rforal}{\,\,\,{\rm for\,\,\,all}\,\,\,}
\newcommand{\CA}{$C^*$-algebra}
\newcommand{\SCA}{$C^*$-subalgebra}

\newcommand{\af}{{\alpha}}

\newcommand{\beq}{\begin{eqnarray}}
\newcommand{\eneq}{\end{eqnarray}}
\newcommand{\tforal}{\,\,\,\text{for\,\,\,all}\,\,\,}

\newcommand{\Qw}{\overline{QT(A)}^w}
\newcommand{\LAff}{{\rm LAff}}
\newcommand{\cuapprox}{\stackrel{\approx}{\sim}}
\newcommand{\Wlog}{Without loss of generality}
\renewcommand{\phi}{\varphi}

\newcommand{\Her}{\mathrm{Her}}
\newcommand{\Cu}{\mathrm{Cu}}
\newcommand{\diag}{{\rm diag}}
\newcommand{\wilog}{without loss of generality}
\newcommand{\wtd}{\widetilde}
\newcommand{\wh}{\widehat}
\newcommand{\simle}{\stackrel{\sim}{<}}

\title{Strict comparison and stable rank one}
\author{Huaxin Lin\\
 }
\date{}
\begin{document}

\maketitle

\begin{abstract}
Let $A$ be a $\sigma$-unital finite simple \CA\, which 
 has strict comparison property.
  We show that if the canonical map 
$\Gamma$ from the Cuntz semigroup to certain lower semi-continuous affine functions is surjective,
then $A$ has tracial approximate oscillation zero and stable rank one.
Equivalently, if $A$ has an  almost unperforated and almost divisible 
Cuntz semigroup, then $A$ has stable rank one and tracial approximate oscillation zero. 
%The converse is known to be true..
%Conversely, if  $A$  has strict comparison and  has tracial approximate oscillation zero, then 
%it is known that $\Gamma$ is surjective. 
\end{abstract}

\section{Introduction}
Let ${\cal Z}$ be the Jiang-Su algebra,  a unital separable infinite dimensional 
simple amenable \CA\, with a unique tracial state such that its ordered $K$-theory 
is exactly the same as that of the complex field $\C$ (\cite{JS}). 
A \CA\, $A$ is ${\cal Z}$-stable, if $A$ is isomorphic to $A\otimes {\cal Z}$ as \CA s. 
%As a \CA, ${\cal Z}$
%is isomorphic to ${\cal Z}\otimes {\cal Z}.$  
It turns out  that the only unital  ${\cal Z}$-stable \CA\, in the UCT class which has 
the same properties that ${\cal Z}$ has as mentioned above is ${\cal Z}$ itself
(up to isomorphism).
%This separable infinite dimensional simple  amenable \CA\, 
%${\cal Z}$ 
%is in fact the unique unital separable  amenable ${\cal Z}$-stable \CA\, (up to isomorphism)  in the UCT class which has the said properties.
This is a result of the classification of separable amenable simple \CA s (see  \cite{GLNII} and \cite{EGLN}, for example). 
If $A$ is a separable simple \CA\, with weakly unperforated $K_0(A),$ then the Elliott invariant 
of $A$ and $A\otimes {\cal Z}$ are exactly the same (see \cite{GJS}).  So the current program of classification 
of separable amenable simple \CA s is mainly for those ${\cal Z}$-stable \CA s.
% (i.e., $A\otimes {\cal Z}\cong A$).

Suppose that $A$ is a unital  separable simple \CA\, with non-trivial 2-quasi-traces. 
Let $a\in  A$ be a positive element, one defines its dimension function 
$\widehat{[a]}(\tau)=d_\tau(a)=\lim_{n\to\infty} \tau(a^{1/n})$ for any 2-quasi-trace $\tau.$
In the case that $A=M_n,$ the algebra of $n\times n$ matrices,
and $\tau$ is the tracial state of $M_n,$ $d_\tau(a)$ is 
just the normalized rank of $a.$
Denote by $QT(A)$ the set of all normalized 2-quasi-traces of $A.$ 
For each $a=(a_{i,j})_{n\times n}\in A\otimes M_n$ and $\tau\in QT(A),$  define $\tau(a):=
(\tau\otimes {\rm Tr})(a)=\sum_{i=1}^n \tau(a_{i,i}),$
where ${\rm Tr}$ is the (non-normalized) trace on $M_n.$ 
We say that $A$ has (Blackadar's) strict comparison (for positive elements), 
if, for any pair of positive elements $a, b\in A\otimes M_r$ (for all $r\in \N$), $a\lesssim b,$ whenever  $d_\tau(a)<d_\tau(b)$
for all $\tau\in QT(A).$  Here $a\lesssim b$ means that there is a sequence of 
elements $x_n\in A\otimes M_r$ such that $a=\lim_{n\to\infty} x_n^*bx_n$ 
(see formal Definition \ref{Dqtr}). Blackadar's strict comparison can also  be defined for 
non-unital simple \CA s.   One of the important features of ${\cal Z}$-stable \CA s is that
they all have strict comparison (see \cite{Ror04JS}).  It is one of the important regularity properties for simple \CA s. 

A part of the Toms-Winter  conjecture states  that the converse also holds for separable amenable 
simple \CA s, i.e., if $A$ is a separable simple amenable
\CA\, with strict comparison, then $A$ is ${\cal Z}$-stable.
  In fact this is the only remaining unsolved part of the Toms-Winter conjecture
  (see \cite{Winter-Z-stable-02}, \cite{MS}, \cite{KR2}, \cite{S-2}, 
\cite{TWW-2}, \cite{Zw}, \cite{Th}, \cite{CETW}, \cite{Linzstable}
%\cite{LinGamma} 
and \cite{LinOZ}, for example). 
Return to the \CA\, $A$ mentioned above, let $a, b\in A$ be positive elements.
We write $a\sim b$ (or $a$ and $b$ are  Cuntz equivalent) if $a\lesssim b$ and $b\lesssim a.$ 
The Cuntz equivalence  relation $ ``\sim"$ is an equivalence relation. 
The Cuntz semigroup ${\rm Cu}(A)$ is the equivalence classes of positive elements 
in the stabilization of $A.$
%Let ${\rm Cu}(A)$ be the Cuntz semigroup of $A.$ 
Consider the canonical map $\Gamma: {\rm Cu}(A)\to \LAff_+( QT(A)),$ i.e., 
the set of strictly positive lower semi-continuous affine functions on $QT(A)$
together with the zero function,
defined naturally by $\Gamma([a])(\tau)=d_\tau(a)$ for $\tau\in QT(A)$ and 
$[a]\in {\rm Cu}(A).$  This map  $\Gamma$ can also be defined for non-unital simple \CA s. 
It is shown by Elliott, Robert and Santiago that if $A$ is  a  unital ${\cal Z}$-stable simple \CA,  then the map $\Gamma$ is always 
surjective (a version  of this also holds for non-unital case, see \cite{ERS}, also remark after Definition \ref{Dgamma} below).   Surjectivity of $\Gamma$ may also be regarded as a regularity condition 
for separable simple \CA. 
One may ask the question whether a separable amenable simple 
\CA\, $A$ with strict comparison and surjective  $\Gamma$ is in fact ${\cal Z}$-stable.  

The notion of  stable rank was introduced to \CA\, theory by Marc  Rieffel  in \cite{Rff}.
 % from  the Bass stable rank.
 A unital \CA\, has stable rank one if 
 %and only if 
 its invertible elements are dense in $A.$ 
 %Simple \CA s with stable rank one 
 The notion plays an important role in the study of simple \CA s
 (for  some earlier work,  see, for example, \cite{Pstr} and \cite{DNNP}). 

M. R\o rdam showed that a unital finite  simple ${\cal Z}$-stable \CA\, has stable rank one (see \cite{Ror04JS}).
L. Robert later showed that any stably projectionless simple ${\cal Z}$-stable \CA\, has almost stable rank one (see \cite{Rlz}).  
Recently,  in \cite{FLL}, it is shown that every finite simple ${\cal Z}$-stable \CA\, actually has stable rank one
(see \cite[Corollary 6.8]{FLL}).
At this moment, we do  not even know, in general,  a separable amenable finite simple 
\CA\, with strict comparison  has stable rank one.  
%If $A$ has stable rank one, and $a, b\in A$ are positive elements such that 
%$a\lesssim b,$ then, $a\simle b,$ i.e.,  there exists $x\in A$ such that $x^*x=a$ and $xx^*\in \overline{bAb}.$
On the other hand,  it was shown by Thiel in \cite{Th}  that if $A$ is a separable  unital  infinite dimensional simple \CA\, with stable rank one, 
then $\Gamma$ is always surjective. The  unital condition is later removed  (see \cite[Theorem 7.14]{APRT}). 
%If $A$ has stable rank one, and $a, b\in A$ are positive elements such that 
%$a\lesssim b,$ then, $a\simle b,$ i.e.,  there exists $x\in A$ such that $x^*x=a$ and $xx^*\in \overline{bAb}.$

Let $A$ be a $\sigma$-unital simple
\CA\, with strict comparison which is not purely infinite.
In \cite{FLosc}, we show, among other things,  that, if $A$ also 
has T-tracial approximate oscillation zero{\footnote{
In the case that $A$ has strict comparison, that $A$ has $T$-tracial approximate oscillation zero
is the same as  that $A$ has tracial approximate oscillation zero (see Definition \ref{DOT} and 
\cite[Definition 4.6]{FLosc}).}}, then $A$ has stable 
rank one. In the converse direction, if the canonical map $\Gamma$ is surjective and $A$  has almost stable rank one,
then $A$ has tracial approximate oscillation zero. 
As  a consequence, under the condition that $\Gamma$ is surjective (and $A$ also has strict comparison), 
that $A$ has stable rank one and that $A$ has almost stable rank one are the same.
However, in this paper,  
we now show that, under the same condition that $\Gamma$ is surjective, 
$A$ {\it always} has stable rank one  and tracial approximate oscillation zero.
% when $\Gamma$ is surjective. 

The main result of this paper may be stated as follows:

\begin{thm}\label{MT}
Let $A$ be a $\sigma$-unital non-elementary simple \CA\, with strict comparison (which is not purely infinite).
Then the following are equivalent:

(1) The  canonical map $\Gamma: \Cu(A)\to {\rm LAff}_+(\wtd{QT}(A))$ 
is surjective. 

(2) $A$ has tracial approximate oscillation zero. 

Moreover, if  (1) or (2) hold, then $A$ has stable rank one. 
\end{thm}

Combining with  \cite{Th} and \cite[Theorem 7.14]{APRT} 
(which show that if $A$ is a separable simple \CA\, with  stable rank one then $\Gamma$ is surjective)
as well as \cite[Theorem 1.1]{FLosc},  we obtain the following:

\begin{cor}\label{CM}
Let $A$ be a separable non-elementary  simple \CA\, with strict comparison (which is not purely infinite).
Then the following are equivalent:

(1) The  canonical map $\Gamma: \Cu(A)\to {\rm LAff}_+(\wtd{QT}(A))$ 
is surjective. 

(2) $A$ has tracial approximate oscillation zero. 

(3) $A$ has stable rank one.

(4) $A$ has property (TM) {\rm (see  \cite[Theorem 1.1]{FLosc}).}

\end{cor}
%%%%%%%%%%%
\iffalse
Let $A$ be a $\sigma$-unital simple \CA. Denote by $V(A)$ the sub-semigroup of ${\rm Cu}(A)$
consisting these elements which can be represented by projections. 
If $A$ has strict comparison and the canonical map $\Gamma: {\rm Cu}(A)\to  {\rm LAff}_+(\wtd{QT}(A))$
is surjective, then,  ${\rm Cu}(A)=V(A)\setminus \{0\}\sqcup  {\rm LAff}_+(\wtd{QT}(A))$
(see \cite[Theorem 6.2]{TTm}). 
\fi
%%%%%%%%%%%%%

%Recall that a \CA\, $A$ is said to be pure if its Cuntz semigroup
%${\rm Cu}(A)$ is almost unperforated and almost divisible. 

We also have the following:

\begin{cor}\label{CM2}
Let $A$ be a non-elementary  finite separable simple \CA\, with 
almost unperforated and almost divisible ${\rm Cu}(A).$
%with almost unperforated  and almost divisible 
%${\rm Cu}(A).$ 
Then $A$ has stable rank one and tracial approximate oscillation zero
\noindent
{\rm (see also Corollary \ref{Cpure}).} 
\end{cor}
%\noindent

This paper is organized as follows.  Section 2  is a preliminary. 
In Section 3,  we consider the following question:
Let $A$ be a $\sigma$-unital simple \CA\, with strict comparison.
Suppose that $a, b\in A_+$ and $d_\tau(a)<d_\tau(b)$ for all nonzero 
2-quasi-traces $\tau.$ Does it imply that there is $x\in A$ such 
that $x^*x=a$ and $xx^*\in \overline{bAb}\,?$  We  show that the answer is yes if $d_\tau(a)$ 
is continuous on $\wtd{QT}(A)$
(see Theorem \ref{Tcomplessapprox}).  This result will be used in Section 4. 
In Section 4,  we show that if $A$ is a $\sigma$-unital simple \CA\, with strict comparison and 
surjective $\Gamma,$ then $A$ has tracial approximate oscillation zero and stable rank one. We also 
give the proof of Theorem \ref{MT} and its corollaries.  
 %Moreover, we give an affirmative answer 
 %to the question in section 3 mentioned above with additional condition that $\Gamma$ is surjective. 

\iffalse
The first problem that we consider is the following comparison question:
Let $A$ be a $\sigma$-unital 

It was recently shown that, if $A$ is a $\sigma$-unital simple \CA\, 
with strict comparison and surjective canonical map $\Gamma,$ then $A$ has stable rank one 
if and only if, for any pair of positive elements $a, b\in A\otimes {\cal K},$
 $a\sim b$  implies that there exists $x\in A\otimes {\cal K}$ such that $x^*x=a$ and $\overline{xx^*Axx^*}
 =\overline{bAb}.$ 
 \fi
 
 \vspace{0.2in}
 
{\bf Acknowledgments}
The author was partially supported by a NSF grant (DMS-1954600). He would like to acknowledge the support during his visits
to the Research Center of Operator Algebras at East China Normal University
which is partially supported by Shanghai Key Laboratory of PMMP, Science and Technology Commission of Shanghai Municipality ( 
\#22DZ2229014).

%
%%%%%%%%%%%%%%
\iffalse
%\pagebreak

Let me try the following:  

1) Let us assume that $\phi, \psi: A\to Q\otimes Q$ 
and $j: Q\otimes Q\to Q\otimes 1$ be a unital endomorphism.

Consider $\Phi=j\circ \phi$ and $\Psi=j\circ \psi.$

 Suppose that  there is $u\in Q\otimes Q$ and $L: C([0,1],Q)$ be 
 such that
 \beq
 \pi_0\circ L=\Phi \andeqn \pi_1\circ L={\rm Ad}\, u\circ \Psi.
 \eneq
(with all other properties). 
Let $\eta>0.$
Recall $\Phi=j\circ \phi.$
There is a continuous path of unitaries $\{w(t): t\in [-1, 0]\}$  in $Q\otimes Q$ such that 
$w(0)=1$ and 
\beq
\|w(-1)^*\Phi(a)w(-1)-\phi(a)\|<\eta \rforal  a\in H_1
\eneq
(for any given finite subset $H_1$).
Define $L(t)(a)=f(t)\phi(a)+(1-f(t))(w^*(-1) \Phi(a)w(1))$ for $a\in A$ and  for 
$t\in [-2,-1]$ so that $f(-2)=1$ and $f(-1)=0$ and $f$ linear.
Also  $L(t)={\rm Ad} w^*(t) \Phi w(t)$ if $t\in [-1, 0].$

For the other end, there exists $V\in Q\otimes Q$ such that
\beq
\|V^*\psi(a)V-{\rm Ad}\, u\circ \Psi(a)\|<\eta \rforal a\in H_1
\eneq
Define, for $t\in [1, 2],$ 
$L(t)(a)=(t-1)V^*\psi(a)V+(2-t){\rm Ad}\, u\circ \Psi(a)$
for $a\in A$ and $t\in [1, 2].$ 
\fi
%%%%%%%%%%%%

\section{Preliminaries}

\begin{df}\label{D1}
Let $A$ be a \CA.
Denote by $A^{\bf 1}$ the closed unit ball of $A,$  and 
by $A_+$ the set of all positive elements in $A.$
Put $A_+^{\bf 1}:=A_+\cap A^{\bf 1}.$
Denote by $\wtd A$ the minimal unitization of $A.$ Let $S\subset A$   be a subset of $A.$
Denote by
${\rm Her}(S)$
the hereditary $C^*$-subalgebra of $A$ generated by $S.$

 Let  ${\rm Ped}(A)$ denote
the Pedersen ideal of $A,$ ${\rm Ped}(A)_+:= {\rm Ped}(A)\cap A_+$
and ${\rm Ped}(A)_+^{\bf 1}={\rm Ped}(A)\cap A_+^{\bf 1}.$ 

Unless otherwise stated (only for Pedersen ideals), an ideal of a \CA\, is {\it always} a closed and  two-sided ideal.
\end{df}

\begin{nota}
Throughout the  paper,
the set of all positive integers is denoted by $\N.$ 
%{{The set of all compact operators on a separable 
%infinite-dimensional Hilbert 
%space is denoted by ${\cal K}.$}} 
%
Let  $A$ 
be a normed space and ${\cal F}\subset A$ be a subset. For any  $\epsilon>0$ and 
$a,b\in A,$
we  write $a\approx_{\epsilon}b$ if
$\|a-b\|< \epsilon$.
We write $a\in_\ep{\cal F}$ if there is $x\in{\cal F}$ such that
$a\approx_\ep x.$

%Let $A$ be a \CA\ and $x\in A.$ Let $|x|:=(x^*x)^{1/2}.$

\end{nota}

\begin{df}
Recall that a \CA\, $A$ has stable rank one if  invertible elements 
of $\wtd A$ are dense in $\wtd A.$ 
\end{df}

\begin{df}\label{Dcuntz}
Denote by ${\cal K}$ the \CA\, of compact operators on $l^2$ and 
by $\{e_{i,j}\}$ a system of matrix unit for ${\cal K}.$ 
Let $A$ be a \CA. 
Consider $A\otimes {\cal K}.$  In  what follows, we identify $A$ with $A\otimes e_{1,1}$ 
as a hereditary \SCA\, of $A\otimes {\cal K},$ whenever it is convenient. 
Let 
%$
%A$ be a \CA\, and 
$a,\, b\in (A\otimes {\cal K})_+.$ 
We write $a\vspace{0.02in}\simle b,$ if there exists $x\in A\otimes {\cal K}$ such that
$x^*x=a$ and $xx^*\in \Her(b)=\overline{b(A\otimes {\cal K})b}.$

We 
write $a \lesssim b$ if there is 
$x_i\in A\otimes {\cal K}$  for all $i\in \N$ 
such that
$\lim_{i\rightarrow\infty}\|a-x_i^*bx_i\|=0$.
We write $a \sim b$ if $a \lesssim b$ and $b \lesssim a$ both  hold.
%We also write $a\lesssim b$ and $a\sim b,$
%when {{there is no confusion.}}
The Cuntz relation $\sim$ is an equivalence relation.
Set $\Cu(A)=(A\otimes {\cal K})_+/\sim.$
%W(A):=M_{\infty}(A)_+/\sim$.
Let $[a]$ denote the equivalence class of $a$. 
We write $[a]\leq [b] $ if $a \lesssim  b$.
%We write $a\simle b,$ if there exists $x\in A\otimes {\cal K}$ such that
%$x^*x=a$ and $xx^*\in \Her(b).$
\end{df}

\begin{df}\label{Dff}
For $\ep>0,$ define a continuous function $f_\ep\in C([0, \infty))$
as follows: $0\le f_\ep(t)\le 1;$
$f_\ep(t)=0$ if $t\in [0,\ep/2],$ $f_\ep(t)=1$ if $t\in [\ep, \infty)$ and 
$f_\ep(t)$ is linear in $[\ep/2, \ep].$
\end{df}

\begin{df}\label{Dqtr}
{{Let $A$ be a $\sigma$-unital \CA. 
A densely  defined 2-quasi-trace  is a 2-quasi-trace defined on ${\rm Ped}(A)$ (see  Definition II.1.1 of \cite{BH}). 
Denote by ${\widetilde{QT}}(A)$ the set of densely defined quasi-traces 
on %of 
$A\otimes {\cal K}.$  
%Suppose, for the convenience of this paper, that all 2-quasi-traces defined on ${\rm Ped}(A)$ 
%are traces. 
%
%Denote by ${\widetilde{T}}(A)$ the set of densely defined traces 
%on %of 
%$A\otimes {\cal K}.$  
 %In what follows we will identify 
 Recall that we identify 
$A$ with $A\otimes e_{1,1}.$
%$ whenever it is convenient. 
Let $\tau\in {\widetilde{QT}}(A).$  Note that $\tau(a)\not=\infty$ for any $a\in {\rm Ped}(A)_+\setminus \{0\}.$

We endow ${\widetilde{QT}}(A)$ 
{{with}} the topology  in which a net 
%$\tau_i$
${{\{}}\tau_i{{\}}}$ 
 converges to $\tau$ if 
 %$\tau_i(a)$ 
${{\{}}\tau_i(a){{\}}}$ 
 converges to $\tau(a)$ for all $a\in 
 {\rm Ped}(A)$ 
 (see also (4.1) on page 985 of \cite{ERS}).
  Denote by $QT(A)$ the set of normalized 2-quasi-traces of $A$ ($\|\tau|_A\|=1$).
  % and (for $r>0$) 
 %$QT_{[0,r]}(A)=\{r\tau: \tau \in QT(A): r\in [0,r]\}$  and 
 %$QT_{(0,r]}(A)=\{r\tau: \tau\in QT(A): r\in (0,r]\}.$ 
 \iffalse
{\red{A  convex subset $S\subset {\wtd{QT}}(A)\setminus \{0\}$ is a basis
 for ${\wtd{QT}}(A),$ if for any $t\in {\wtd{QT}}(A)\setminus \{0\},$ there exists a unique  pair $r\in \R_+$ and 
 $s\in S$ such that $r\cdot s=t.$ }}
Let $e\in {\rm Ped}(A)_+\setminus \{0\}$ be a full element of $A.$
 Then $S_e=\{\tau: \tau\in {\wtd{QT}}(A): \tau(e)=1\}$ is a Choquet simplex and is a basis 
 for the cone ${\wtd{QT}}(A)$ (see Proposition 3.4 of \cite{TTm}).
 \fi
 %%%%%%%%%
% {\red{By Theorem 4.4 of \cite{ERS}, 
% if $A$ is separable, this topology is second countable. 
% Suppose that $T\subset {\widetilde{QT}}(A)$ is a compact subset.  Then 
% $T$ is metrizable.}} 

% We  will write  $\tau(a)$ for $\tau|_{{\rm Ped}(A)}(a)$ for any $a\in A.$
%{\green{(do you mean $a\in {\rm Ped}(A)?$)}}
%We will continue to use $\tau$ for its extension on $M_n(A)$
%for each $n\ge 1.$ 
}}

Note that, for each $a\in ({{A}}%A_+
\otimes {\cal K})_+$ and $\ep>0,$ $f_\ep(a)\in {\rm Ped}(A\otimes {\cal K})_+.$ 
Define 
\beq
\widehat{[a]}(\tau):=d_\tau(a)=\lim_{\ep\to 0}\tau(f_\ep(a))\rforal \tau\in {\widetilde{QT}}(A).
\eneq

Let $A$ be a simple \CA\,with $\wtd{QT}(A)\setminus \{0\}\not=\emptyset.$
% whose  2-quasi-traces are all traces.   
Then
$A$
is said to have (Blackadar's) strict comparison, if, for any $a, b\in (A\otimes {\cal K})_+,$ 
condition 
\beq
d_\tau(a)<d_\tau(b)\rforal \tau\in {\widetilde{QT}}(A)\setminus \{0\}
\eneq
implies that 
$a\lesssim b$ (see the paragraph after  Definition  \ref{DpureW}). 
%
%
%Except in subsection \ref{sub2}, we
%
%
\end{df}

\begin{df}\label{DGamma}
Let $A$ be a \CA\, with ${\wtd{QT}}(A)\setminus\{0\}\not=\emptyset.$
Let $L({\wtd{QT}}(A))$ be the set of all continuous real  functions 
on ${\wtd{QT}}(A)$ 
such that $f(t+\tau)=f(t)+f(\tau)$ and $f(s\cdot \tau)=sf(\tau)$
for all $s\in \R_+$ and $t, \tau\in {\wtd{QT}}(A).$  Note that if $f\in L({\wtd{QT}}(A))$ then 
$f(0)=0.$
%$f(s\cdot \tau)=sf(\tau)$ since $0\in {\wtd{QT}}(A).$ 
Let $S\subset {\wtd{QT}}(A)$ be a convex subset. 
Set 
%(if $0\not\in S,$ we ignore the condition $f(0)=0$)
\beq
\Aff_+(S)&=&\{f\in L({\wtd{QT}}(A)): f \,\,{\rm affine},
%C(S, \R)_+: f \,\,% {\rm linear},
%{{\rm affine}}, 
f(s)>0\,\,{\rm for}\,\,s\not=0,\,\,\}\cup \{0\},\\
%
%{\rm LAff}_f(S)_+&=&\{f:S\to [0,\infty): \exists \{f_n\}, f_n\nearrow f,\,\,
 %f_n\in \Aff(S)_+\},\\
%{\rm LAff}_{f,+}(S)&=&\{f:S\to [0,\infty): \exists \{f_n\}, f_n\nearrow f,\,\,
 %f_n\in \Aff_+(S)\},\\
%{\rm LAff}(S)_+&=&\{f:S\to [0,\infty]: \exists \{f_n\}, f_n\nearrow f,\,\,
% f_n\in \Aff(S)_+\},\\
{\rm LAff}_+(S)&=&\{f:S\to [0,\infty]: \exists \{f_n\}, f_n\nearrow f,\,\,
 f_n\in \Aff_+(S)\}.
 \eneq
Note that if $0\in S,$ $f(0)=0.$
For a  simple \CA\, $A$ and   each $a\in (A\otimes {\cal K})_+,$ the function $\hat{a}(\tau)=\tau(a)$ ($\tau\in S$) 
is in general in ${\rm LAff}_+(S).$   If $a\in {\rm Ped}(A\otimes {\cal K})_+,$
then $\wh{a}\in \Aff_+(S).$
For 
$\widehat{[a]}(\tau)=d_\tau(a)$ defined above,   we have 
$\widehat{[a]}\in {\rm LAff}_+({\wtd{QT}}(A)).$  Note that 
$\wh{a}$ and $\wh{[a]}$ are not the same in general.
\end{df}

\begin{df}\label{Dgamma}
Let   $\Gamma: \Cu(A)\to {\rm LAff}_+({\wtd{QT}}(A))$ be 
the canonical map defined by $\Gamma([a])(\tau)=\widehat{[a]}(\tau)=d_\tau(a)$ 
for all $\tau\in {\wtd{QT}}(A).$
\end{df}

It is helpful to notice that, when $A$ is exact, every 2-quasi-trace is a trace (\cite{Haagtrace}). 
By \cite[Theorem 6.6 and 4.4]{ERS}, if $A$ is ${\cal Z}$-stable, then $\Gamma$ is surjective.

(1) In the case that $A$ is  algebraically simple (i.e., $A$ is a simple \CA\,  and $A={\rm Ped}(A)$), 
$\Gamma$ also induces a canonical map 
$\Gamma_1: \Cu(A)\to {\rm LAff}_+(\Qw),$
where $\Qw$ is the weak*-closure of $QT(A).$  Since, in this case, 
$\R_+\cdot \Qw={\wtd{QT}}(A),$ the map $\Gamma$ is surjective if and only if $\Gamma_1$
is surjective.  We would like to point out that, in this case, $0\not\in \Qw$ (see Lemma 4.5 of \cite{eglnp}
and Proposition 2.9 of \cite{FLosc}).

(2) In the case that $A$ is stably finite and simple,
denote by $\Cu(A)_+$ the set of purely non-compact elements (see Proposition 6.4 of \cite{ERS}).
Suppose that $\Gamma$ is surjective.  Then  $\Gamma|_{\Cu(A)_+}$ is surjective as well  (see 
Theorem 6.6 of \cite{ERS} and the remark (2) after Definition 2.13 of \cite{FLosc}).

(3) Let $V(A)$ be the sub-semigroup of ${\rm Cu}(A)$ consisting 
of those elements which can be represented by projections in ${\rm Ped}(A\otimes {\cal K}).$
If $A$ has strict comparison and $\Gamma$ is surjective, then one has a nice description 
of ${\rm Cu}(A):$ ${\rm Cu}(A)=V(A)\sqcup ({\rm LAff}_+(\wtd{QT}(A))\setminus \{0\})$ 
(see \cite[Theorem 6.2]{TTm}), where the order and the mixed addition are  defined in 
the paragraph above Theorem 6.2 of   \cite{TTm} (see also subsection 6.1 of \cite{Rl}). 
To be more precise, $[p]+f:=\Gamma([p])+f;$ $[p]\le f$ if and only if 
$\tau(p)<f(\tau)$ for all $\tau\in \wtd{QT}(A)\setminus \{0\},$ while $f\le [p]$ if and only 
if $f(\tau)\le \tau(p)$ for all $\tau\in \wtd{QT}(A)$ (assuming $p\in A\otimes {\cal K}$ is a projection).
It should also be added that, if $p, q\in A\otimes {\cal K}$ are projections 
such that $\tau(p)<\tau(q)$ for all $\tau\in \wtd{QT}(A)\setminus \{0\},$ 
then $h:=\wh{[q]}-\wh{[p]}\in\Aff_+(\wtd{QT}(A))\setminus \{0\}.$ 
Hence $[p]\le [p]+(1/2)h\le [q].$   If ${\rm Cu}(A)=V(A)\sqcup ({\rm LAff}_+(\wtd{QT}(A))\setminus \{0\})$
as above,  then 
%In particular, in this case,
$A$ has strict comparison.
If, in addition,  $A={\rm Ped}(A),$ 
one may write  ${\rm Cu}(A)=V(A)\sqcup ({\rm LAff}_+(\Qw)\setminus \{0\}).$

%It is helpful to notice that, when $A$ is exact, every 2-quasitrace is a trace (\cite{Hu}). 
 %It is also helpful to note that, if $A$ is algebraically simple and $T(A)$ (or $QT(A)$) 
 %is compact, then $T(A)$ (or $QT(A)$) is a compact basis for the cone $\wtd T(A).$ 
 %It follows that $T(A)$  (or $QT(A)$)
 %it is a metrizable Choquet simplex (see \cite[Theorem 3.1]{PdMIII}). 

%{TOstosurj-1}
%
%%%%%%%%%%%%%%%
%
%%%%%%%%%%%%%%%%%%%%%%%%%%%%%%%%%%%%%%%%
%
%%%%%%%%%%%%%%%%%%
%For most part of the paper,  we  will assume that all 2-quasi-traces of a separable \CA\, $A$ 
%in this paper  
%are in fact traces for convenience. This is the case if $A$ is exact (by \cite{Haagtrace}).
%\end{df}

\begin{df}\label{Dalmostunp}
Let $A$ be a \CA\, and ${\rm Cu}(A)$ its Cuntz semigroup. 
If $x, y\in {\rm Cu}(A),$ we write $x\ll y,$ if   whenever $y\le \sup_n y_n$
for some increasing sequence $\{y_n\},$ then there exists $n_0\in \N$ such  that $x\le y_{n_0}.$

We say ${\rm Cu}(A)$ is almost unperforated if $(k+1) x\le ky$ 
implies $x\le y$ for all $x, y\in S$ and $k\in \N$ (see \cite[Definition 3.1]{Ror04JS}).

We say ${\rm Cu}(A)$ is almost divisible, if for all $x, y\in {\rm Cu}(A)$ with 
$x\ll y,$ and $n\in \N,$ there exists $z\in {\rm Cu}(A)$ such that $nz\le y$ and $x\le (n+1)z$
(see Property (D) in the proof of \cite[Proposition 6.2.1]{Rl}). 
\end{df}

\begin{df}\label{DpureW}
Recall that a \CA\, is said to be {\it pure} if ${\rm Cu}(A)$ is almost unperforated and almost divisible 
(\cite[Definition 3.6 (i)]{Winter-Z-stable-02} and \cite[Subsection 6.3]{RS}). 
\end{df}

Suppose that $A$ is a $\sigma$-unital simple \CA\, 
with $\wtd{QT}(A)=\{0\}.$  Given  any $a, b\in (A\otimes {\cal K})_+\setminus \{0\}.$
Then, one always has  $d_\tau(a)<d_\tau(b)$ for all $\tau\in \wtd{QT}(A)\setminus \{0\}=\emptyset.$
One may assume that $A$ has strict comparison means that $a\lesssim b$ 
in this case.  This implies that
$A$ is purely infinite.  However, one may also exclude the case that $A$ is purely infinite
in the next statement. 

The following  proposition is a combination of  results of R\o rdam, Elliott-Robert-Santiago, Robert, 
 Tikuisis-Toms   and Winter, and others
(of some earlier versions).
We state here for  convenience and some clarification (cf. \cite{Ror04JS}, \cite{ERS}, \cite{Winter-Z-stable-02} and 
\cite{TTm} and \cite{BH}) for the later statements.

\begin{thm}\label{PCuntz}
Let $A$ be a separable simple \CA\,
% which is not purely infinite. 
Then the following are equivalent:

(1) $A$ has strict comparison and the canonical  map $\Gamma$ is surjective;

(2)  ${\rm Cu}(A)=V(A)\sqcup ( {\rm LAff}_+(\wtd{QT}(A))\setminus \{0\});$

(3) ${\rm Cu}(A)$ is almost unperforated and almost divisible, and

(4) $A$ is pure. 
\end{thm}

\begin{proof}
 (1) $\Leftrightarrow$ (2): Note that, as mentioned in comment (3) after Definition \ref{Dgamma},   (1) $\Rightarrow$ (2)  follows from the same proof of \cite[Theorem 6.2]{TTm}. 
 That (2) $\Rightarrow$ (1)  follows from the definition of the order on $V(A)\sqcup
 ({\rm LAff}_+(\wtd{QT}(A))\setminus \{0\}$  (see (3) after Definition \ref{Dgamma}). 
 
 (3) $\Leftrightarrow$ (4) follows from the definition. 
 
 To see (2) $\Rightarrow$ (3), let  $a, b\in (A\otimes {\cal K})_+$   be
 %and 
 %$[, y\in {\rm Cu}(A)$ 
 such that $(k+1)[a]\le k [b]$  for some $k\in \N.$  Then $\wh{[a]}(\tau)<\wh{[b]}(\tau)$ for 
 all $\tau\in \wtd{QT}(A)\setminus \{0\}.$ It follows that $[a]\le [b].$ So ${\rm Cu}(A)$ is almost unperforated.
 
 To see that ${\rm Cu}(A)$ is almost divisible, 
 let $x, y\in {\rm Cu}(A)$ and $x\ll y.$  If $x\in {\rm LAff}_+(\wtd{QT}(A)),$ then, for any $n\in \N,$
 choose $z=(1/n)x.$ Then $nz\le y$ and $x\le (n+1)z.$
 If $x=[p]$ for some projection $p,$ then $\Gamma([p])\in \Aff_+(\wtd{QT}(A)).$
 The assumption that $[p]\ll y$ implies that $\wh{y}(\tau)-\wh{[p]}\in \Aff_+(\wtd{QT}(A))\setminus \{0\}.$
 Put $h=(1/2)(\wh{y}(\tau)-\wh{[p]}).$ 
 Then $x+h\le y.$ Now $x+h\in {\rm LAff}_+(\wtd{QT}(A))\setminus \{0\}.$ Choose $z=(1/n)(x+h).$
 Then $nz\le y$ and $x\le (n+1)z.$

 For  (3) $\Rightarrow$ (1),   
 we first note that $A$ has strict comparison if ${\rm Cu}(A)$ is almost unperforated. 
 This follows from \cite[Proposition 3.2]{Ror04JS}{\footnote{Strictly  speaking \cite[Proposition 3.2]{Ror04JS}
 implies that strict comparison for elements in $\cup_n M_n(A)_+.$ 
 Let $a, b\in (A\otimes {\cal K})_+$ such that $d_\tau(a)<d_\tau(b)$ for all 
 $\tau\in \wtd{QT}(A)\setminus \{0\}.$ Let $\ep>0$
and $\{e_n\}$ be an approximate identity for $A\otimes {\cal K}$ such that 
$e_n\in M_n(A),$ $n\in \N.$ Since $\wh{[f_\ep(a)]}\le \wh{f_{\ep/2}(a)}$ and the latter is continuous,
for some $m\in \N,$ $\wh{[f_\ep(a)]}<[b^{1/2} e_mb{\wh{^{1/2}]\,\,}}$  on $\wtd{QT}(A)\setminus \{0\}.$
Then $[e_n^{1/2}f_{\ep}(a)^{1/2}{\wh{e_n^{1/2}]\,\,}}<[e_m^{1/2}b\wh{e_m^{1/2}]\,\,}$ on $\wtd{QT}(A)\setminus \{0\}.$
Hence $e_n^{1/2}f_\ep(a)e_n^{1/2}\lesssim e_m^{1/2} b e_m^{1/2}\sim b^{1/2} e_m b^{1/2}.$ 
It follows that there are $x_n\in A\otimes {\cal K}$ such that
$\|x_n^*x_n-e_n^{1/2}f_\ep(a)e_n^{1/2}\|<1/2^n$ and $x_nx_n^*\in \Her(b^{1/2} e_m b^{1/2})\subset \Her(b).$
 Since $\lim_{n\to\infty}\|e_n^{1/2}f_\ep(a)e_n^{1/n}-f_\ep(a)\|=0,$ we conclude that
 $\lim_{n\to\infty}x_n^*x_n=f_\ep(a).$ Thus $f_\ep(a)\lesssim b$ for any $\ep>0.$ Hence $a\lesssim b.$}}
 %(see 
 %Suppose  that $A\otimes {\cal K}$ contains a nonzero 
 %projection $p.$ By Brown's stable isomorphism theorem (\cite{Br1}), 
 %$pAp\otimes {\cal K}\cong A\otimes {\cal K}.$ So we may assume that $A$  is unital. 
 %It follows from \cite[Corollary 4.7]{Ror04JS} 
 %(see also \cite[Corollary 4.6]{Ror04JS}).
 %that $A$ has strict comparison.
 %Now suppose that $A$ is stably projectionless.   Then $A$ has strict comparison follows from 
 %the statement of
 (see also \cite[Theorem 6.6]{ERS}). 

 Since $A$ now has strict comparison and 
 ${\rm Cu}(A)$ is almost divisible, by \cite[Corollary 5.8]{Rl},  $\Gamma$ is surjective (see also the second 
 paragraph of the proof of \cite[Proposition 6.2.1]{Rl}). 
 This proves (3) $\Rightarrow$ (1).
 %
 %It is also clear that (2) $\Rightarrow$ (3). 
% 
\end{proof}
%  a{\footnote{aa}}

\begin{df}\label{Domega}
Let $A$ be a \CA\, with ${\widetilde{QT}}(A)\setminus \{0\}\not=\emptyset.$ 
Let $S\subset {\widetilde{QT}}(A)$
%\setminus \{0\}$ 
be a compact subset. 
Define, for each $a\in (A\otimes {\cal K})_+,$
\beq
\omega(a)|_S&=&\lim_{n\to\infty}\sup\{d_\tau(a)-\tau(f_{1/n}(a)): \tau\in S\}
\eneq
(see \cite[Definition 4.1]{FLosc} and A1 of \cite{eglnkk0}).  We will assume that $A$ is algebraically simple  and 
only consider the case that $S=\Qw,$ and in this case, we will write 
$\omega(a)$ instead of $\omega(a)|_{\Qw},$ in this paper.  It should be mentioned 
that $\omega(a)=0$ if and only if $d_\tau(a)$ is continuous on $\Qw.$

If $A$ is a $\sigma$-unital simple \CA,  choose $e\in {\rm Ped}(A)_+\setminus \{0\}.$
Then $\Her(e)$ is algebraically simple.  Since $A\otimes {\cal K}\cong \Her(e)\otimes {\cal K},$
we may write $\omega(a)=0$ if $\omega(a)|_S=0,$ where $S=\overline{QT(\Her(e))}^w.$
As in \cite{FLosc} (see the last paragraph of \cite[4.1]{FLosc}), this does not depend on the choice of $e.$ 
\end{df}

\begin{df}\label{DOT}
Let $A$ be a \CA\, with ${\widetilde{QT}}(A)\setminus \{0\}\not=\emptyset.$ 
Let $S\subset {\widetilde{QT}}(A)$
%\setminus \{0\}$ 
be a compact subset and 
%Let  $A$ be an  algebraically simple \CA\, and 
$x\in A\otimes {\cal K}.$ Define
\beq
\|x\|_{_{2, S}}=\sup\{\tau(x^*x)^{1/2}: \tau\in S\}.
\eneq
In this case,  for $a\in (A\otimes {\cal K})_+,$ 
%$a\in {\rm Ped}(A\otimes {\cal K})_+^{\bf 1},$ 
we write $\Omega^T(a)|_S=0,$ if $\|a\|_{_{2, S}}<\infty$ and, 
if there is a sequence $b_n\in {\rm Ped}({\rm Her}(a))_+^{\bf 1}$ such that, with 
%$S=\overline{QT({\rm Her}(a))}^w,$ 
\beq
\lim_{n\to\infty}\omega(b_n)|_S=0\andeqn \lim_{n\to\infty}\|a-b_n\|_{_{2, S}}=0
\eneq
%where $\|x\|_{_{2, \Qw}}=\sup\{\tau(x^*x)^{1/2}: \tau\in \Qw\}$ 
(see Proposition 4.8 of \cite{FLosc}).
 
 Now let us assume that $A$ is algebraically simple and $QT(A)\not=\emptyset.$ 
 Let $a\in 
 %{\rm Ped}(
 (A\otimes {\cal K})_+^{\bf 1}.$ 
 In this case we write $\Omega^T(a)=0$ if $\Omega^T(a)|_{_{\Qw}}=0.$

Even if $A$ is not algebraically simple (but $\sigma$-unital and simple), 
we may fix a  nonzero element $e\in {\rm Ped}(A)\cap  A_+$ (so that ${\rm Her}(e)\otimes {\cal K}\cong 
A\otimes {\cal K}$) and choose $S=\overline{QT({\rm Her}(e))}^w.$ 
Then, for any 
$a\in (A\otimes {\cal K})_+,$   we write $\Omega^T(a)=0$ if $\Omega^T(a)|_S=0$ 
(this does not depend on the choice of $e,$ see 5.1 of \cite{FLosc}).

A $\sigma$-unital simple \CA\, $A$ is said to have $T$-tracial approximate oscillation zero, if
%, there is $e\in {\rm Ped}(A)_+$ 
%such that  
$\Omega^T(a)=0$ for 
every  positive element $a\in {\rm Ped}(A\otimes {\cal K})_+.$  
There is also a notion of tracial approximate oscillation zero.
In the case that $A$ has strict comparison, $A$ has tracial approximate oscillation zero 
if and only if 
$A$ has $T$-tracial approximate oscillation zero. Since we consider only 
\CA s with strict comparison in this paper, for the convenience, 
we will only use the term tracial approximate oscillation zero. 

If we view $\|\cdot \|_{_{2, \Qw}}$ as an  $L^2$-norm, 
 then  that $A$ has T-tracial approximate oscillation zero has an analogue  to that ``almost'' continuous functions are 
 dense in the $L^2$-norm.
%where $S=\overline{QT({\rm Her}(e))}^w.$
\end{df}

\begin{df}\label{Hild}
Let $A$ be a \CA.
%%%%%%%%%%%%
%\iffalse
Denote by $H_A,$ or $l^2(A),$ the standard countably generated Hilbert 
$A$-module
$$
H_A=\{\{a_n\}: a_n\in A, \sum_{n=1}^k a_n^*a_n \,\,\,{\rm converges\,\,\,
in\,\,\, norm\,\, as}\,\, k\to\infty\},
$$
where the inner product is defined by
$
\la \{a_n\}, \{b_n\}\ra=\sum_{n=1}^{\infty} a_n^*b_n.
$
%\fi
%%%%%%%%%%%%%%%%
Let $H$ be a countably generated Hilbert $A$-module. Denote by
$B(H)$ the Banach algebra of all bounded $A$-module maps on $H$ and 
denote by $L(H)$ the subalgebra of $B(H)$ whose maps have adjoints. 
For $x, y\in H,$ define $\theta_{x, y}\in L(H)$ by 
$\theta_{x,y}(z)=x\la y,z\ra$ for all $z\in H.$ 
Denote by $F(H)$ the linear span of those module maps with the form
$\theta_{x, y},$ where $x, y\in H.$ Denote by $K(H)$ the
closure of $F(H).$  $K(H)$ is a \CA.   Then $K(l^2(A))\cong A\otimes {\cal K}.$ 
It follows from a result of
Kasparov (\cite{K}) that $L(H)=M(K(H)),$  the multiplier algebra of
$K(H),$ and, by \cite{Lnbd}, $B(H)=LM(K(H)),$ the left multiplier
algebra of $K(H).$

Let $A$ be a $\sigma$-unital \CA\, and $H$ be a countably generated Hilbert $A$-module.
Then $K(H)$ is a $\sigma$-unital hereditary \SCA\,(see \cite[Lemma 2.13]{Lninj}) of $K(l^2(A))\cong A\otimes {\cal K}.$
Let $a\in K(H)$ be a strictly positive element. Then $\overline{a(l^2(A))}=H.$ 
For each $a\in (A\otimes {\cal K})_+=K(l^2(A))_+,$ define 
$H_a=\overline{a(l^2(A))}.$ If $b\in (A\otimes {\cal K})_+$ such 
that $H_a=H_b,$ then $b\in K(H_a)$ and $\overline{bH_a}=H_a.$ 
It follows that $b$ 
is also  a strictly positive element of $K(H_a)$ (see  \cite[Lemma 13]{MP}).
\end{df}

 Let $a, b\in (A\otimes {\cal K})_+.$
 If $A$ has stable rank one, it was shown in \cite{CEI} 
that $H_a$ is unitarily equivalent to $H_b$
(as Hilbert $A$-modules) if and only if $[a]=[b]$ in ${\rm Cu}(A).$ 

\begin{df}\label{Dmodule}
Let $A$ be a $\sigma$-unital simple \CA\, and 
$H_a$ be a countably generated Hilbert $A$-module associated with the positive element
$a\in A\otimes {\cal K}.$ Define 
\beq
d_\tau(H_a)=d_\tau(a)\rforal \tau\in {\wtd{QT}}(A).
\eneq
One checks that $d_\tau(H_a)$ does not depend the choice of  $a$ as a strictly positive element 
in $K(H_a).$ 
Write  $\omega(H_a)=0$ if $\omega(a)=0.$ 
\end{df}

\section{Strict comparison}
Let $A$ be a $\sigma$-unital simple \CA\, with strict comparison.
We consider the following question:

{\bf Q}: 
Suppose that $a, b\in (A\otimes {\cal K})_+$ such that
\beq
d_\tau(a)<d_\tau(b)\rforal \tau\in {\wtd{QT}}(A)\setminus \{0\}.
\eneq
%One may ask whether 
Does it imply that
 $a\simle b$  (see Definition \ref{Dcuntz})?

Suppose that 
$H_a$ and $H_b$ are countably generated Hilbert $A$-modules
(where $a, b\in (A\otimes {\cal K})_+$). 
 The question above is the same as the following:
%We consider  the following question: 

Suppose that 
$d_\tau(H_a)<d_\tau(H_b)$ for all $\tau\in \wtd{QT}(A)\setminus \{0\}.$ 
When is  $H_a$ unitarily equivalent to a Hilbert $A$-submodule of $H_b?$ 

%This is equivalent to ask whether  that $d_\tau(a)<d_\tau(b)$ for all $\tau\in \wtd{QT}(A)$ 
%implies that $a\simle b.$   
If $A$ is ${\cal Z}$-stable, then the answer to {\bf Q}
is affirmative. In fact the answer  to {\bf Q} is affirmative if $A$ has stable rank one. 

 We will partially answer the question in this section.

We will repeatedly use the following fact.

\begin{lem}\label{Lcomapctcontain}
Let $\Delta$ be a compact space and $g, f_n\in C(\Delta)$ be continuous functions 
on $\Delta$
such that $g(t)>0,\, f_n(t)>0$ for all $t\in \Delta,$ $n\in \N.$
Suppose that $f_n\le f_{n+1}$ for all $n\in \N$ and 
$g(t)<\lim_{n\to\infty}f_n(t)$ for all $t\in \Delta.$
Then there exists $n_0\in \N$ such that, for all $n\ge n_0,$
\beq
g(t)<f_n(t)\tforal t\in \Delta.
\eneq
\end{lem}
%
%%%%%%%%%%%%%%%%%%%%%%%%%%%%%%%%
%
\iffalse
%
%%%%%%%%%%%%%%%%%%%%%
\begin{proof}
For each $t_0\in \Delta,$ there exists $n(t_0)\in \N$ such that
$g(t_0)<f_{n(t_0)}(t_0).$ Define $h(t)=f_{n(t_0)}(t)-g(t)\in C(\Delta).$
Since $h(t_0)>0.$ By the continuity, there is a neighborhood $O(t_0)\subset \Delta$
such that 
$$
h(t)>0\tforal t\in O(t_0).
$$
Since $\Delta$ is compact, there are $t_1, t_2,...,t_m\in \Delta$ and 
$n(t_i)\in \N,$ $1\le i\le m$ such that
\beq
f_{n(t_i)}(t)-g(t)>0\rforal t\in O(t_i),\,\,1\le i\le m\andeqn \cup_{i=1}^m O(t_i)\supset \Delta.
\eneq
Choose $n_0=\max\{n(t_i): 1\le i\le m\}.$ 
If $n\ge n_0,$ since $f_n\le f_{n+1}$ for all $n,$
we have, for each $t\in O(t_i),$ $1\le i\le m,$ 
\beq
f_n(t)\ge f_{n(t_i)}(t)>g(t).
\eneq
Since $\cup_{i=1}^m O(t_i)\supset \Delta,$ $f_n(t)>g(t)$ for all $t\in \Delta$ and $n\ge n_0.$
\end{proof}
%
\fi
%%%%%%%%%%%%%%%%%%

%Now let us consider simple \CA\, $A$ with strict comparison. 
%If $A$ is ${\cal Z}$-stable or at least $A$ has stable rank one, then 
%$A$ is stable if $d_\tau(a)=\infty$ (for all $\tau\in {\widetilde{QT}}(A)$) for some strictly positive element
%$a$ of $A.$ 

%Recall that a

\begin{lem}{\rm{[cf.  Lemma 2.3 of \cite{Lin91cs}]}}\label{Lcontsc}
Let $A$ be a non-elementary simple \CA\, and $a\in A_+\setminus \{0\}.$
Then there exists a sequence of mutually orthogonal  elements $\{b_n\}$ 
in $\Her(f_\dt(a))_+^{\bf 1}$ for some $0<\dt<1/2$ such that $\|b_n\|=1$ and 
\beq
2^n [b_{n+1}]\le  [f_\dt(b_n)]\,\,\, { in}\,\, {\rm Cu}(A),\,\,\, n\in \N.
\eneq
\end{lem}

\begin{proof}
\Wlog, we may assume that $\|a\|=1.$
Choose $0<\dt<1/2$ such that $f_\dt(a)\not=0.$
Since $\Her(f_\dt(a))$ is a simple and non-elementary, by p.67 of \cite{AS},
there exists $c_0\in \Her(f_\dt(a))_+^{\bf 1}$ such that ${\rm sp}(c_0)=[0,1].$
Hence $\Her(f_\dt(a))_+^{\bf 1}$ contains $2$ mutually orthogonal elements 
$a_{0,1}, a_{0,2}$ with $\|a_{0,i}\|=1,$ $i=1,2.$  Note that $a_{0,i}\in {\rm Ped}(A)_+$
and  that $f_\dt(a_{0,i})\not=0,$ $i=1,2.$ 
By Lemma 3.4 of \cite{eglnp},  there are $x_{1,1}, x_{1,2},..,x_{1,m_1}\in A$ 
such that
\beq
f_\dt(a_{0,1})=\sum_{i=1}^{m_1} x_{1,i}^* a_{0,2}x_{1,i}. 
\eneq
We may assume that $c_1'=x_{1,1}^*a_{0,2}x_{1,1}\not=0.$
Then 
\beq
c_1'\lesssim a_{0,2}\andeqn c_1'\le f_\dt(a_{0,1}).
\eneq
Put $c_{1,1}=c_1'/\|c_1\|.$  Then $c_{1,1}\lesssim a_{0,1}, a_{0,2}.$
It follows that
\beq
2[c_{1,1}]\le [f_\dt(a)]\andeqn c_{1,1}\in \Her(f_\dt(a_{0,1}))_+.
\eneq
Repeating the same argument, we obtain  $y_{1,1},y_{1,2},...,y_{1,m_1'}\in A$ such that
\beq
f_\dt(a_{0,2})=\sum_{i=1}^{m_1'} y_{1,i}^* c_{1,1}y_{1,i}.
\eneq
We may assume  that $c_2'=y_{1,1}^*c_{1,1}y_{1,1}\not=0.$
Then 
\beq
c_2'\lesssim c_{1,1}\andeqn c_2'\in \Her(f_\dt(a_{0,2})).
\eneq
Put $c_{1,2}=c_2'/\|c_2'\|.$ 
Then 
\beq
[c_{1,2}]\le [c_{1,1}],\,\, 2[c_{1,2}]\le [f_\dt(a)]\andeqn c_{1,2}\perp c_{1,1}.
\eneq

Next we note that $f_\dt(c_{1,2})\not=0.$
By repeating  the same argument above, we obtain  two mutually orthogonal non-zero elements
$c_{2,1}, c_{2,2}\in \Her(f_\dt(c_{1,2}))_+$ with $\|c_{2,i}\|=1,$ $i=1,2,$ and 
\beq
c_{2,2}\lesssim c_{2,1}, \,\, 2[c_{2,1}]\le [c_{1,2}]\,\,\, {\rm in}\,\, {\rm Cu}(A).
\eneq 
Put $c_1=c_{1,1},$ $c_2=c_{2,1}.$ Note that 
\beq
c_1\perp c_2,\,\, 2[c_1]\le [f_\dt(a)]\andeqn 2[c_2]\le [f_\dt(c_1)].
\eneq
By repeated  applying the argument above, we obtain a 
sequence of mutually orthogonal elements $\{c_n\}$ 
with $\|c_n\|=1$ 
such that 
$$
2[c_{n+1}]\le [f_{\dt}(c_n)].
$$
The lemma then follows by choosing 
$b_n=c_{2^n},$ $n\in \N.$
\end{proof}

\begin{lem}\label{Ltwo}
Let $A$ be a $\sigma$-unital non-elementary  algebraically simple \CA\,
%$a,\, b\in (A\otimes {\cal K})_+$
such that $QT(A)\not=\emptyset.$  Let $a,\, b\in (A\otimes {\cal K})_+.$
%Suppose that
%\beq
%d_\tau(a)<d_\tau(b)\tforal \tau\in \Qw.
%\eneq
If one of the following holds:

(1) $d_\tau(a)+\omega(a)<d_\tau(b)$ for all $\tau\in \Qw,$  or

(2) there exists $r>0$ such that $d_\tau(a)+r<d_\tau(b)$ for all $\tau\in \Qw,$

\noindent
then there are two mutually orthogonal nonzero elements $b_0, b_1\in \Her(b)_+^{\bf 1}$
such that
\beq
d_\tau(a)<d_\tau(b_1)\tforal \tau\in \Qw.
\eneq
\end{lem}

\begin{proof}
We first note that the case (1) may be reduced to case (2).
To see this we may assume that 
 $\omega(a)=0.$ In other words, $\wh{[a]}\in \Aff_+(\Qw).$ 
Then  $h(\tau)=d_\tau(b)-d_\tau(a)\in \LAff_+(\Qw).$
Put 
\beq 
r_0=(1/2) \inf\{h(\tau): \tau\in \Qw\}>0.
\eneq
If $r_0<\infty,$ choose $r=r_0$ and, if $r_0=\infty,$ choose  any  $0<r<\infty.$ 
Then 
\beq
d_\tau(a)+r<d_\tau(b)\rforal \tau\in \Qw.
\eneq

Next we show the  conclusion holds for case (2). 
Choose $\eta>0$ such that $f_\eta(b)\not=0.$ 
Then $f_\eta(b)\lesssim b$   and $f_{\eta}(b)\in {\rm Ped}(A\otimes {\cal K})_+.$ 
Hence $\sup\{d_\tau(f_\eta(b)):\tau\in \Qw\}<\infty$ (see, for example,  (2) of Proposition 2.10 of \cite{FLosc}).
Choose $n\in \N$  with $n\ge 2$ such that 
\beq
\sup\{d_\tau(f_\eta(b)): \tau\in \Qw\}/n<r/2.
\eneq
By Lemma \ref{Lcontsc}, one obtains an element $c\in \Her(f_\eta(b))_+\subset \Her(b)_+$ with $\|c\|=1$
such that
\beq
n[c]\le [f_\eta(b)].
\eneq
Hence, for all $\tau\in \Qw,$
\beq
d_\tau(c)<r/2.
\eneq
%We $\stackrel{\sim}{<}.$
Choose $b_0=f_{1/4}(c).$ Then $\|b_0\|=1$ and $d_\tau(b_0)<r/2.$ 
Define $b_1=(1-f_{1/8}(c))^{1/2} b(1-f_{1/8}(c))^{1/2}.$
Then 
\beq
b_1b_0=(1-f_{1/8}(c))^{1/2} b(1-f_{1/8}(c))^{1/2}b_0=0=b_0b_1.
\eneq
We compute that
\beq
b\lesssim b^{1/2}(1-f_{1/8}(c))b^{1/2}\oplus 
%(1-f_{1/8}(c))^{1/2}
f_{1/8}(c).
\eneq
Hence 
\beq
d_\tau(b)\le d_\tau(b_1)+d_\tau(c)\le d_\tau(b_1)+r/2\rforal \tau\in \Qw.
\eneq
It follows that
\beq
d_\tau(b_1)\ge d_\tau(b)-r>d_\tau(a)\rforal \tau\in \Qw.
\eneq
\end{proof}

\begin{cor}\label{Ctwo}
Let $A$ be a $\sigma$-unital non-elementary  algebraically simple \CA, $a,\, b\in (A\otimes {\cal K})_+$
such that $QT(A)\not=\emptyset.$ 
If 
$$
d_\tau(a)<d_\tau(b)\tforal \tau\in \Qw\andeqn \omega(a)=0,
$$
then there are mutually orthogonal non-zero elements $b_1, b_0\in \Her(b)_+^{\bf 1}$ 
and $0<\dt<1/4,$ 
such that
\beq
d_\tau(a)<d_\tau(f_\dt(b_1))\tforal \tau\in \Qw.
\eneq
\end{cor}

\begin{proof}
By Lemma \ref{Ltwo}, there are mutually orthogonal non-zero elements $b_1, b_0\in \Her(b)_+^{\bf 1}$
such that 
\beq
d_\tau(a)<d_\tau(b_1)\rforal \tau\in \Qw.
\eneq
Since $d_\tau(a)$ is continuous on $\Qw,$ by Lemma \ref{Lcomapctcontain},
there is $n\in \N$ such that
\beq
d_\tau(a)<\tau(f_{1/n}(b_1))<d_\tau(f_{1/n}(b_1))\rforal \tau\in \Qw.
\eneq
Choose $\dt=1/n.$
\end{proof}

\begin{thm}\label{Tcomplessapprox}
Let $A$ be a $\sigma$-unital simple \CA\, with strict comparison which is not purely infinite.
Let $a, b\in (A\otimes {\cal K})_+$ be such that 
\beq\label{lessapprox-1}
d_\tau(a)<d_\tau(b)<\infty \rforal \tau\in {\widetilde{QT}}(A)\setminus \{0\}.
\eneq
Suppose that $\omega(a)=0.$ Then,  $a \simle b.$
\end{thm}
(see Corollary \ref{Ccomparison}.)

\begin{proof}
We may assume that $A$ is non-elementary and $\|a\|=1=\|b\|.$

Pick a nonzero element $e\in {\rm Ped}(A)$ with $0\le e\le 1$ and $C={\rm Her}(e).$
Then $C$ is algebraically simple. By Brown's stable isomorphism theorem (\cite{B1}), 
$C\otimes {\cal K}\cong A\otimes {\cal K}.$ \Wlog, we may assume that $A$ is itself 
algebraically simple. By  \cite[Proposition 2.9]{FLosc}, $0\not\in \Qw.$ 
%%%%%%%%%%%%%%%%%%%%%%%%%%%%%
\iffalse
Since $\omega(a)=0,$ the function $d_\tau(a)$ is continuous on $\Qw.$
Then, by \eqref{lessapprox-1}, $g(\tau)=d_\tau(b)-d_\tau(a)\in \LAff(\Qw)_{++}$ and nonzero. 
Since $A$ is a non-elementary simple \CA,  every nonzero hereditary \SCA\, contains a sequence 
of mutually orthogonal nonzero positive elements. Therefore, 
one obtains a nonzero element 
$d\in {\rm Her}(b)_+$ such that $d_\tau(d)<(1/2)g(\tau)$ for all $\tau\in \Qw.$ 
We may assume that $\|d\|=1.$ 
Choose $d_0=f_{1/16}(d),$ $d_{00}=f_{1/4}(d)\not=0.$  and $b_0=(1-d_0)^{1/2}b(1-d_0)^{1/2}.$
Then $b_0d_{00}=d_{00}b_0=0.$  
Since 
\beq
b_0\lesssim b^{1/2}(1-d_0)b^{1/2}\oplus b^{1/2}d_0b^{1/2}\andeqn\\
b^{1/2}d_0b^{1/2}\cuapprox d_0^{1/2}bd_0^{1/2}\lesssim d_0,
\eneq
we conclude that
\beq
d_\tau(a)<d_\tau(b_0)\rforal \tau\in \Qw.
\eneq
Since ${\rm Her}(d_{00})$ is a nonzero hereditary non-elementary simple \SCA, 
one obtains a sequence of mutually orthogonal nonzero elements $\{d_{0,k}\}_{k\in \N}$ in 
${\rm Her}(d_{00}).$ 
\fi
%%%%%%%%%%%%%%%%%%
By Lemma \ref{Ltwo} and also Lemma \ref{Lcomapctcontain},
 there are non-zero elements $b_0, b_1'\in \Her(b)_+$ 
such that, for some $0<\eta<1/4,$  
\beq\label{34}
b_0b_1'=b_1'b_0=0\andeqn d_\tau(a)<d_\tau(f_{\eta}(b_1'))\rforal \tau\in \Qw.
\eneq
We may assume that $\|b_1'\|=\|b_0\|=1.$
Choose $b_1=f_{\eta}(b_1').$

Applying \ref{Lcontsc},  there are mutually orthogonal elements
$\{b_{0,n}\}\in \Her(b_0)_+$ with $\|b_{0,n}\|=1,$  $i=1,2,....$ 
such that $2^n[b_{0,n+1}]\le [f_\dt(b_{0,n})],$ $n\in\N$ (for some $0<\dt<1/4$).
Note that, since $\|b_{0,n}\|=1,$ $\tau(b_{0,n})\le d_\tau(b_{0,n})$
for all $\tau\in \Qw,$ $n\in \N.$
Put
\beq\label{sigma-n}
\sigma_n=(1/2)\inf\{\tau(b_{0,n}): \tau\in \Qw\}>0,\,\,n\in \N.
\eneq

Since $\omega(a)=0,$
choose $0<\dt_1'<1/4$ such that 
\beq\label{36}
\sup\{d_\tau(a)-f_{\dt_1'}(a): \tau\in \Qw\}<\sigma_1. 
\eneq
%Choose $\dt_1=\dt_1'/4.$
Now since $A$ has strict comparison, by \eqref{34}, $a\lesssim b_1.$
%f_{\eta_1}(b_1)$ 
%for some $\eta_1>0.$
%In other words, there are $z_{1,n}\in A\otimes {\cal K}$  such
%that
%$$
%z_{1,n}^*z_{1,n}\to a\andeqn z_{1,n}z_{1,n}^*\in \overline{f_{\eta_1}(b_1)}.
%$$
%By Lemma 2.2 of \cite{Rr1}, 
%%%%%%%%%%%%%%%%%%%%%%%%%
%
%
%%%%%%%%%%%%%%%%%%%%%%%
%for any $0<\dt<1/2,$ $f_\dt(a)\lessapprox b_0.$
%\fi
%%%%%%% 
%Since $\omega(a)=0,$ we may choose $0<\dt_1'<1/4$ such
%that 
%\beq
%\sup\{d_\tau(a)-f_{\dt_1'}(a): \tau\in \Qw\}<(1/2)d_\tau(d_{0,1})\rforal \tau\in \Qw.
%\eneq
Put $\dt_1=\dt_1'/8.$ 
%Choose $a_1=f_{\dt_1}(a)$ and $a_1'=f_{\dt_1/2}(a).$   
By applying \cite[Proposition 2.4, (iv)]{Rr1}, we obtain  $x_1\in A\otimes {\cal K}$ such 
that
%\fi
%%%%%%%%%%%%%%%%
%\iffalse
\beq
x_1^*x_1=f_{\dt_1/4}(a)\andeqn x_1x_1^*\in 
% {\rm Her}(f_{\eta_1}(b_1))\subset 
{\rm Her}(b_1).
\eneq
%%%%%%%%%%%%%%%%%%%%%
%
\iffalse
%%%%%%%%%%%%%%%%%%%%%%%% 
Let $x_1=v_1|x_1|$ be the polar decomposition of $x_1$ in $A^{**}.$ 
Then, by ?, the map $\phi: \Her(x_1^*x_1)\to \Her(x_1x_1^*)$ defined by
$\phi(c)=v_1^*cv_1$ for all $c\in \Her(x_1^*x_1).$
Note that $f_{\dt_1/2}(a)\in \Her(x_1^*x_1).$
Define $c_1=\phi(f_{\dt_1/2}(a))=v_1^*f_{\dt_1/2}(a)v_1$ and 
$c_1'=\phi(f_{\dt_1/4}(a))=v_1^*f_{\dt_1/4}(a)v_1.$ 
%For any $g'\in \Her(b_1)_+,$ put $g=(1-c_1')^{1/2} b_1(1-c_1')^{1/2}
Define 
\beq
b_2'=(1-c_1')^{1/2} f_{\eta_1/2}(b_1)(1-c_1')^{1/2}=f_{\eta_1/2}(b_1')-c_1'.
\eneq
Then
\beq
c_1b_2'=c_1(1-c_1')^{1/2} b_1(1-c_1')^{1/2}=0=b_2'c_1.
\eneq
So $b_2'\in \Her(b_1)^\perp\cap \Her(b_1).$ 
We compute that
\beq
\tau(b_2')&=&\tau(f_{\eta_1/2}(b))-\tau(c_1')\ge d_\tau(f_{\eta_1}(b))-\tau(f_{\dt_1/2}(a))\\
&>&d_\tau(a)-d_\tau(f_{\dt_1}(a))\hspace{0.8in} \rforal \tau\in \Qw.
\eneq
\fi
%%%%%%%%%%%%%%%%%%%%%%%%%%%%%%%%%%%%
%
%%%%%%%%%
Put  
%$a_1=f_{\dt_1/4}(a)$ and 
$c_1=x_1x_1^*.$ 
Choose $0<\dt_2'<\dt_1/16$ such that (recall that $\omega(a)=0$)
\beq
\sup\{d_\tau(a)-f_{\dt_2'}(a): \tau\in \Qw\}<\sigma_2.
\eneq
Put $\dt_2=\dt_2'/8.$ 
Define $g_2\in C([0, \infty)]$ such that $0\le g_2(t)\le 1,$
$g_2(t)=1,$ if $t\in [\dt_2/4, \dt_1'/4],$ $g_2(t)=0,$ if $t\in [0, \dt_2/8]\cup [\dt_1'/2, \infty],$ 
and $g_2(t)>0$ if $t\not\in [0, \dt_2/8]\cup [\dt_1'/2, \infty].$ 
Define $a_{1,0}=f_{\dt_1'}(a),$ $a_1=f_{\dt_1/4}(a)$ and $a_2'=g_2(a).$ 
Note that (see Definition \ref{Dff})
\beq\label{Tcomplessapprox-10}
a_{1,0}\perp a_2'\andeqn a_2'\in \Her(f_{\dt_2/4}(a))_+.
\eneq
%It follows that
Moreover (by also \eqref{36})
\beq
d_\tau(a_2')<\sigma_1\rforal \tau\in \Qw.
\eneq
Since $A$ has strict comparison, by \eqref{sigma-n} and  Lemma 2.4  (iv) of \cite{Rr1}, there is $x_2\in A\otimes {\cal K}$  such that
\beq
x_2^*x_2=f_{1/16}(a_2')\andeqn x_2x_2^*\in \Her(b_{0,1}).
\eneq
Put $a_2=x_2^*x_2$ and $c_2=x_2x_2^*.$ 
Since, for each $t\in [\dt_2/4, \dt_1'/4],$ $f_{1/16}(g_2(t))=1,$ we compute that 
\beq
f_{\dt_1}(t)+f_{1/16}(g_2(t))\ge 1 \rforal t\in [\dt_2/4, \infty).
\eneq
We then   have 
\beq
f_{\dt_2/2}(a)\le a_1+a_2\le 2.
\eneq
Denote by $C^*(a)$ the commutative \SCA\, generated by 
a single element $a.$ 
By \eqref {Tcomplessapprox-10},
\beq
a_1+a_2\in C^*(a)\cap \Her(f_{\dt_2/4}(a)).
\eneq
We also have  (by \eqref{Tcomplessapprox-10} and by the fact that $b_1b_{1,0}=b_{1,0}=b_1$)
that  
\beq
a_{1,0}\perp a_2\andeqn
c_1\perp c_2.
\eneq

Suppose that we have constructed $a_1, a_2,..., a_m\in  C^*(a)_+^{\bf 1}\subset \Her(a)_+^{\bf 1},$ $c_1,c_2,...,c_m,$
$x_1, x_2,...,x_m\in A\otimes {\cal K},$ 
decreasing positive numbers $\{\dt_1', \dt_2',...,\dt_m'\}$ and $\{\dt_1, \dt_2, ...,\dt_m\}$ 
with $\dt_i=\dt_i'/8,$ $\dt_{i+1}'<\dt_i/8,$ $1\le i\le m,$ such that

(i) $c_1\in \Her(b_1)_+^{\bf 1},$ $c_{i+1}\in \Her(b_{0,i})_+^{\bf 1},$ $1\le i\le m-1;$

(ii) $f_{\dt_i/2}(a)\le \sum_{j=1}^m a_j\le 2,$ $1\le i\le m;$

(iii) $f_{\dt_i'}(a)\perp a_{i+1},$ $a_ia_{j}=0,$ $|j-i|\ge 2,\, 1\le i,j\le m,$
 and $\sum_{i=1}^ja_i\in \Her(f_{\dt_j/4}(a))$
for $1\le j\le m;$

(iv) $\sup\{d_\tau(a)-f_{\dt_i'}(a): \tau\in \Qw\}<\sigma_i,\,\,1\le i\le m,$ and 

(v) $x_i^*x_i=a_i\andeqn x_ix_i^*=c_i,$ $1\le i\le m.$

Since $\omega(a)=0,$ choose $0<\dt_{m+1}'<\dt_m/8$ such that
\beq
\hspace{-1.25in}{\rm (iv')}\hspace{0.4in}&&\hspace{0.4in}\sup\{d_\tau(a)-\tau(f_{\dt_{m+1}'}(a)): \tau\in \Qw\}<\sigma_{m+1}.
\eneq
Put $\dt_{m+1}=\dt_{m+1}'/8.$  Define $g_{m+1}\in C([0, \infty)]$ such that $0\le g_{m+1}(t)\le 1,$
$g_{m+1}(t)=1,$ if $t\in [\dt_{m+1}/4, \dt_{m}'/4],$ $g_{m+1}(t)=0,$ if $t\in [0, \dt_{m+1}/8]\cup [\dt_{m}'/2, \infty],$ 
and $g_{m+1}(t)>0$ if $t\not\in [0, \dt_{m+1}/8]\cup [\dt_m'/2, \infty].$ 
Define $a_{m,0}=f_{\dt_{m}'}(a)$ 
%$a_{m+1=f_{\dt_1/4}(a)$ 
and $a_{m+1}'=g_{m+1}(a).$ 
Note that
\beq
a_{m,0}\perp a_{m+1}'\andeqn a_{m+1}'\in \Her(f_{\dt_{m+1}/4}(a))_+.
\eneq
Moreover (by (iv))
\beq
d_\tau(a_{m+1}')<\sigma_m\rforal \tau\in \Qw.
\eneq
Since $A$ has strict comparison, by \eqref{sigma-n} and Lemma 2.4 (iv) of \cite{Rr1}, there is $x_{m+1}\in A\otimes {\cal K}$  such that
\beq
x_{m+1}^*x_{m+1}=f_{1/16}(a_{m+1}')\andeqn x_{m+1}x_{m+1}^*\in \Her(b_{0,m}).
\eneq
Put $a_{m+1}=x_{m+1}^*x_{m+1}$ and $c_{m+1}=x_{m+1}x_{m+1}^*.$ 
Note that $a_1, a_2,...,a_{m+1}\in C^*(a)_+^{\bf 1}$ and $c_1, c_2,...,c_m, c_{m+1}$
are mutually orthogonal, and 

(i') $c_1\in \Her(b_1)_+^{\bf 1},$ $c_{i+1}\in \Her(b_{0,i})_+^{\bf 1},$ $1\le i\le m.$

Note that, for each $t\in [\dt_{m+1}/4, \dt_{m}'/4],$ $f_{1/16}(g_{m+1}(t))=1,$  we have 
\beq
f_{\dt_m}(t)+f_{1/16}(g_{m+1}(t))\ge 1 \rforal t\in [\dt_{m+1}/4, \infty).
\eneq
Hence, by (ii), 
\beq\label{II'}
\sum_{i=1}^m a_i+a_{m+1}\ge f_{\dt_m/2}(a)+a_{m+1}\ge f_{\dt_{m+1}/2}(a).
\eneq
Since $\dt_m'<\dt_{m-1}/8$ by (iii) (recall that $g_{m+1}(t)=0$ for all $t\in [\dt_m'/2, \infty)$),
\beq
\sum_{i=1}^{m-1}a_i\perp a_{m+1}\andeqn f_{\dt_m'}(a)\perp a_{m+1}.
\eneq
Hence $a_i a_{m+1}=0$ for any $1\le i\le m-1.$
It follows that, by (iii), 
\beq
\sum_{1\le i\le m+1,even} a_i\le 1\andeqn \sum_{i\le i\le m+1, odd} a_i\le 1.
\eneq
%%%%%%%%%%% 
\iffalse
Recall that $C^*(a)\cong C_0({\rm sp}(a)).$ 
On the support $S\subset {\rm sp}(a)$ of $a_{m+1},$
\beq
(\sum_{i=1}^{m+1}a_i)|_S=(a_m+a_{m+1})|_S\le 2. 
\eneq
If  $S_1\subset {\rm sp}(a)$ and $S_1\cap S=\emptyset,$   then, on $S_1,$ 
\beq
(\sum_{i=1}^{m+1}a_i)|_{S_1}=(\sum_{i=1}^m a_i)|_{S_1}\le 2.
\eneq
\fi
%%%%%%%%%%%%%%%%%%%%
Hence  we have (first one is inequality is the same as \eqref{II'})

(ii') $f_{\dt_{m+1}/2}(a)\le \sum_{i=1}^{m+1} a_i\le 2$ and 

(iii') $f_{\dt_m'}(a)\perp a_{m+1},$ 
$a_i a_{j}=0,$  $|j-i|\ge 2,$ $1\le i\le m+1,$ and  $\sum_{i=1}^{m+1}a_i\in \Her(f_{\dt_{m+1}/4}(a))_+.$

We also have

(v') $x_{m+1}^*x_{m+1}=a_{m+1}$ and $x_{m+1}x_{m+1}^*=c_{m+1}.$ 

By (i') -- (v') above and by induction, we obtain a sequence of elements 
$\{a_m\},$ $\{c_m\},$ $\{x_m\}$ and a sequence of decreasing numbers 
$\{\dt_m'\}$ and $\{\dt_m\}$ such that (i'), (ii'), (iii'), (iv') and (v') hold.

Since $\{b_1, b_{0,1}, b_{0,2},...\}$ is a set of mutually orthogonal positive elements, by (i),
 we have $c_ic_j=0=c_jc_i$ if $i\not=j.$
Define  $c=\sum_{n=1}^\infty c_n/n.$ Then $c\in \Her(b)_+^{\bf 1}.$ 
Put $C=\Her(c)$  and define $E_{c,n}=\sum_{i=1}^n c_i^{1/n},$ $n\in \N.$
Then $\{E_{c,n}\}$ forms an approximate identity for $C.$
%%%%%%%%%%%%%%%%%%%%%%%%%%%%%%%%%%%
\iffalse
   Note that $\sum_{n=m}^\infty c_n\in M(C)$ for all $m\ge 1,$  i.e., 
$\sum_{i=m}^n c_n$ converges  to  $\sum_{i=m}^\infty c_n$ in the strict topology of $M(C)$
(as $n\to\infty$). 
 Then, for any $c'\in C,$ any $\ep>0,$ there exists $n_0\in \N$ such 
that
\beq\label{cC}
\lim_{m\to\infty}\|\sum_{i=m}^\infty c_n c'\|=0=\lim_{m\to\infty} \|c'\sum_{i=m}^\infty c_n\|.
\eneq
\fi
%%%%%%%%%%%%%%%%%%%%%%%%

Next,  
%def\sum_{i=1}^\infty x_i^*.$  Note 
%that,  
since $c_ic_j=0$ if $i\not=j,$ we have 
$x_i^*x_j=0$ if $ i\not=j.$ Hence
\beq
(\sum_{i=1}^nx_i^*)(\sum_{i=1}^nx_i^*)^*=(\sum_{i=1}^nx_i^*)(\sum_{i=1}^nx_i)
=\sum_{i=1}^n x_i^*x_i=\sum_{i=1}^na_i.
%%
%(\sum_{i=1}^n x_i)(\sum_{i=1}^nx_i^*)=\sum_{i=1}^n x_ix_i^*+\sum_{i=1}^{n-1} x_ix_{i+1}^*
%+x_nx_{n-1}^*=\sum_{i=1}^n c_i+\sum_{i=1}^{n-1} x_ix_{i+1}^*
%+x_nx_{n-1}^*.
\eneq
It follows from (ii) that 
$\|\sum_{i=1}^n x_i^*\|\le 2$ for all $n\in \N.$ 
%%%%%%%%%%%%%%%%%%%%
\iffalse
Since $c_1, c_2,...,$ are mutually orthogonal, 
if $i+1\not=j,$ $(x_ix_{i+1}^*)x_jx_j^*=0$ and $j+1\not=i,$ $x_jx_{j+1}^*x_ix_{i+1}^*=0.$
It follows that
\beq
\|\sum_{i=1}^{n-1}x_ix_{i+1}^*\|\le \|\sum_{1\le i\le n-1, even} x_ix_{i+1}^*\|+\|\sum_{1\le i\le n-1, odd}x_ix_{i+1}^*\|\le 2.
\eneq
 Hence $\|\sum_{i=1}^n x_i^*\|\le 4.$
 \fi
 %%%%%%%%%%%%
Moreover, for each fixed $n,$ and any $m, k\in \N$ with $m\ge n,$    by (v),
%(see \eqref{cC}), for any $c'\in C,$ 
\beq
(\sum_{i=1}^m x_i^*) E_{c,n}=(\sum_{i=1}^n x_i^*)(E_{c,n})\andeqn (\sum_{i=m+1}^{m+k}x_i^*)E_{c,n}=0.
%c'\to yc' \,\,\, {\rm in\,\,\, norm}, \,\, {\rm as}\,\, n\to\infty.
\eneq
Recall also that $\lim_{n\to\infty}c'E_{c,n}=\lim_{n\to\infty} E_{c,n}c'=c'$
for all $c'\in C.$
It follows that, for any $c'\in C,$  $\{(\sum_{i=1}^m x_i^*)c'\}$ is Cauchy and 
\beq
\sum_{i=1}^m x_i^*c'\to yc'\,\,\,{\rm (in\,\, norm),}
\eneq
where $y=\sum_{i=1}^\infty x_i^*(\in LM(C, A\otimes {\cal K})).$ 
%In other words, $y\in LM(C),$  the left multiplier algebra of $C.$
%Note that (recall that $c_ic_j=c_jc_i=0$ for all $i\not=j$)
%\beq
%yc(yc)^*=yc^2y^*=\lim_{n\to\infty}\sum_{k=1}^nx_k^*(c_k/k)^2x_k\in \Her(a).
%eneq
%Thus $y\in LM(C, \Her(a)).$  

Write $H_a=\overline{a(A\otimes {\cal K})}$  and $H_c=\overline{c(A\otimes {\cal K})}.$
Then $H_a$ and $H_c$ are Hilbert $A\otimes {\cal K}$-modules (or closed right ideals of $A\otimes {\cal K}$). 
Define  $T: H_c\to H_a$ by $T(z)=yz$ for all $z\in H_c.$   Then 
$T\in B(H_c, H_a)$ (see \cite[Theorem 1.5]{Lnbd}).

For each $n,$ 
\beq
T(\sum_{i=1}^n c_i)^{1/m}(\sum_{i=1}^nx_i)=\sum_{i=1}^n x_i^*c_i^{1/m}x_i\in H_a.
\eneq
Since  $c_i=x_ix_i^*,$  we have that $x_i^*c_i^{1/m}x_i\to x_i^*x_i$ in norm. 
It follows that 
\beq
\sum_{i=1}^n a_i=\sum_{i=1}^n x_i^*x_i\in \overline{T(H_c)}.
\eneq
By (ii) above, 
\beq
f_{\dt_{n+1}/2}(a)\le \sum_{i=1}^n x_i^*x_i. 
\eneq 
Hence 
\beq
f_{\dt_{n+1}/2}(a)\in \overline{(\sum_{i=1}^n a_i)(A\otimes {\cal K})(\sum_{i=1}^n a_i)}\subset \overline{(\sum_{i=1}^n a_i)(A\otimes {\cal K})}
\subset \overline{T(H_c)}.
\eneq
It follows that $a\in \overline{T(H_c)}.$ Consequently 
$T(H_c)$ is dense in $H_a.$ 
Applying \cite[Corollary 3.9]{BL}, we get that $a\simle c.$ Since $c\in \Her(b),$
we finally conclude that 
\beq\nonumber
a\simle b.
\eneq
\end{proof}

\begin{cor}\label{Cnew}
Let $A$ be a $\sigma$-unital simple \CA\, with strict comparison 
and with a non-empty ${\widetilde{QT}}(A)$ which has a basis 
%which contains 
containing finitely many  extremal points.
%which is not purely infinite.
Let $a, b\in (A\otimes {\cal K})_+$ be such that 
\beq\label{lessapprox-1}
d_\tau(a)<d_\tau(b)<\infty \rforal \tau\in {\widetilde{QT}}(A)\setminus \{0\}.
\eneq

Then $a\simle b.$
\end{cor}

\begin{proof}
Let $S\subset {\widetilde{QT}}(A)\setminus \{0\}$ be a basis for the cone 
${\widetilde{QT}}(A)$ and its extremal boundary $\partial_e(S)$ has only finitely many points.
Then $\widehat{[a]}$ must be continuous on $S$ and hence $\omega(a)=0.$
Therefore Theorem \ref{Tcomplessapprox}  applies (in fact, in this case, $A$ has tracial approximate oscillation zero,
and hence $A$ has stable rank one---see  Theorem 1.1 of \cite{FLosc}).
\end{proof}

\begin{NN}\label{NN2}
{\rm Let $A$ be a $\sigma$-unital simple \CA\, with a non-empty $S\subset \wtd{QT}(A)\setminus \{0\}$
such that $\R_+\cdot S=\wtd{QT}(A).$  If $\wtd{QT}(A)\not=\{0\},$ 
such $S$ exists.   For example,  choose $e\in {\rm Ped}(A\otimes {\cal K})_+\setminus \{0\}.$ 
Then $S_e=\{\tau\in \wtd{QT}(A): \tau(e)=1\}$ is a compact convex subset in $\wtd{QT}(A)\setminus \{0\}$
 such that
$\R_+\cdot S_e=\wtd{QT}(A).$ In fact $S_e$ is a basis for the cone $\wtd{QT}(A).$}

 {\rm Fix $S$ as above. Let $J=\{a\in A\otimes {\cal K}:  \sup\{d_\tau(a^*a): \tau\in S\}<\infty\}.$ Then $J$ is a two-sided  ideal of $A$ (not necessarily closed).
 Note that, if $a\in (A\otimes {\cal K})_+,$ then $f_{\ep/2}(a),\,\, f_\ep(a)\in {\rm Ped}(A\otimes {\cal K})$ for any 
 $0<\ep<\|a\|.$
 %$\|a\|>\ep>0.$ 
 Hence $f_\ep(a)\in J.$ 
It follows that $J$ is dense in $A\otimes {\cal K},$ whence 
 $J$ contains ${\rm Ped}(A\otimes {\cal K}).$    One may ask when an element $a\in J$
 is in ${\rm Ped}(A\otimes {\cal K}).$  The next corollary gives a partial answer  (see also Corollary \ref{CPed2}).}
\end{NN}

The following is also used in the next section (see the proof of Lemma \ref{Ldig} and Corollary \ref{CPed2}).

\begin{cor}\label{CPedsen}
Let $A$ be a $\sigma$-unital simple \CA\, with $\wtd{QT}(A)\not=\{0\}$ and with strict comparison, and 
$x\in (A\otimes {\cal K}).$
If $d_\tau(x^*x)<\infty$ for all $\tau\in \wtd{QT}(A)$ and 
$\omega(x^*x)=0.$
Then $x\in {\rm Ped}(A\otimes {\cal K}).$
If  furthermore, $x\in A,$ then $x\in {\rm Ped}(A).$
\end{cor}

\begin{proof}
Put $a=x^*x.$
Choose a nonzero element $e\in A_+.$ 
Consider $b=\diag(e, e/2,...,e/n,...)\in (A\otimes {\cal K})_+.$
Then 
$d_\tau(b)=\infty$ for all $\tau\in \wtd{QT}(A)\setminus \{0\}.$ There exist $\dt>0$ such that
\beq
d_\tau(a)<d_\tau(f_\dt(b))\rforal \tau\in \wtd{QT}(A)\setminus \{0\}.
\eneq
By Theorem \ref{Tcomplessapprox}, there exists $y\in A\otimes {\cal K}$ such that
\beq
y^*y=a\andeqn yy^*\in \overline{f_\dt(b)A f_\dt(b)}.
\eneq
Since $f_\dt(b)\in {\rm Ped}(A\otimes {\cal K})_+,$ 
we have $a\in {\rm Ped}(A\otimes {\cal K})_+.$ Hence $x\in {\rm Ped}(A\otimes {\cal K}).$

If $a\in A_+,$ 
choose $d\in {\rm Ped}(A)_+\setminus \{0\}.$ Then $d\in {\rm Ped}(A\otimes {\cal K})_+.$
Since $a\in {\rm Ped}(A\otimes {\cal K}),$ there are $y_1, y_2,...,y_m\in A\otimes {\cal K}$ such that
$
a=\sum_{i=1}^m y_i^*d y_i.
$
Consider a system of matrix units $\{e_{i,j}\}$ for ${\cal K}.$ 
We may view $\{e_{i,j}\}\subset \wtd A\otimes {\cal K}.$
Then 
\beq
a=e_{1,1}ae_{1,1}=\sum_{i=1}^m (e_{1,1} y_ie_{1,1})^*d (e_{1,1}ye_{1,1}).
\eneq
Since $z_i=e_{1,1}y_ie_{1,1}\in A,$ we conclude that $a\in {\rm Ped}(A),$ whence $x\in {\rm Ped}(A).$
\end{proof}

\begin{cor}\label{Cmodule}
Let $A$ be a $\sigma$-unital simple \CA\, with strict comparison.
Suppose that $H_a$ and $H_b$ are  countably generated Hilbert $A$-modules
such that
\beq
\omega(H_a)=0\andeqn d_\tau(H_a)<d_\tau(H_b)\rforal \tau\in \wtd{QT}(A)\setminus \{0\}
\eneq
(see Definition \ref{Dmodule}). 
Then $H_a$ is unitarily equivalent to a Hilbert $A$-submodule of $H_b.$
\end{cor}

\begin{proof}
This follows from Theorem \ref{Tcomplessapprox} above, and Corollary 3.9 and (1) of Theorem 3.6
of \cite{BL}.
\end{proof}

\section{Tracial approximate oscillation zero}

\begin{lem}\label{Lorth}
Let $A$ be a $\sigma$-unital algebraically simple \CA\, with strict comparison
which is not purely infinite.
% and with a strictly positive element $e\in {\rm Ped}(A)_+^{\bf 1}.$
Suppose that $e,\, a\in {\rm Ped}(A\otimes {\cal K})_+^{\bf 1}$ such that
$\omega(a)=0$ and $d_\tau(a)< d_\tau(e)$
for all $\Qw.$  Then, for any $\ep>0,$ there is 
$b\in \Her(e)_+^{\bf 1}$ such that

(1) $d_\tau(a)-\ep
%\inf\{d_\tau(a): \tau\in \Qw\}
<\tau(b)\le d_\tau(a)$
for all $\tau\in \Qw$ and 

(2) $\omega(b)<\ep/2,$ 

(3) $d_\tau(e)-d_\tau(a)\le d_\tau(e_C)\le \tau(q)
< d_\tau(e)-d_\tau(a) +\ep$ for all $\tau\in \Qw,$

%\inf\{d_\tau(a): \tau\in \Qw\},$
\noindent
where $e_C$ is a positive element in $\Her(b)^\perp\cap \Her(e)$ and 
$q$ is the open projection associated with $\Her(b)^\perp\cap \Her(e).$
\end{lem}

\begin{proof}
Since $d_\tau(a)$ is continuous on $\Qw,$ by Lemma \ref{Lcomapctcontain}, there exists $0<\eta<1/4$ such that
\beq
d_\tau(a)<d_\tau(f_\eta(e))\rforal \tau\in \Qw.
\eneq
By Theorem \ref{Tcomplessapprox}, there exists $x\in A\otimes {\cal K}$ such that
\beq
x^*x=a\andeqn xx^*\in \Her(f_\eta(e)).
\eneq
Since $\omega(a)=0,$ by \cite[Proposition 4.2]{FLosc},
%Definition 4.1 of \cite{FLosc},  
we also have $\omega(xx^*)=0.$ One then may
%$\wh{a_1}=h_1,$ $\omega(x_1x_1^*)=0.$
choose $n(1)\in \N$ such that, for all $n\ge n(1),$  
\beq\label{Ldig-e-5}
d_\tau(xx^*)-\tau(f_{1/n}(xx^*))<\ep/2\rforal \tau\in \Qw.
%\inf\{d_\tau(a): \tau\in \Qw\}. 
\eneq
Put $b=f_{1/2(n(1))}(xx^*)$  and $b'=f_{1/4n(1)}(xx^*).$ 
Since $d_\tau(a)=d_\tau(xx^*)$ for all $\tau\in \Qw,$ we obtain  (1).

Since, for any $0<\eta<1/4,$
\beq
d_\tau(b)-\tau(f_\eta(b))\le d_\tau(xx^*)-\tau(f_{1/n(1)}(xx^*))<\ep/2\rforal \tau\in \Qw,
\eneq
we obtain  that 
$$
\omega(b)<\ep/2.
%\inf\{h_1(\tau): \tau\in \Qw\}.
$$ 
So (2) holds.

Let $p$ be the open projection of $A^{**}$ associated with $\Her(e).$
Define $g\in C([0, \infty])$ with $0<g(t)\le 1$ for all $t\in (0, \infty),$ $g(0)=0,$
$g(t)=1$ if $t\in [\eta/2, \infty).$
Define $e_C=g(e)-b'.$
Then 
\beq
be_C=b(g(e)-b')=b-bb'=0=e_Cb.
\eneq
In other words, $e_C\in \Her(b)^\perp\cap \Her(e).$

Let $p_\eta$ be the spectral projection (in $A^{**}$) of $e$ 
corresponding to the open set $(\eta/2, \|e\|].$ 
Note that $p_\eta b'=b'=b'p_{\eta}.$ 
For each $m\in \N,$
\beq
\hspace{-0.4in}(g(e)-b')^{1/m}&=&((1-p_\eta)g(e))^{1/m}+(p_\eta (g(e)-b'))^{1/m}\\
&\ge & ((1-p_\eta)g(e))^{/1m}+(p_\eta-b').
\eneq
Therefore 
\beq\nonumber
\sup\{\tau(g(e)-b')^{1/m}):\tau\in \Qw\}\ge 
%&\ge&
%&=&((1-p_\eta)g(e))^{1/m}+(p_\eta (g(e)-b'))^{1/m}\\
%&\ge &
\sup\{\tau(((1-p_\eta)g(e))^{/1m}+(p_\eta-b')):\tau\in \Qw\}.
\eneq
Note that 
\beq
((1-p_\eta)g(e))^{/1m}+(p_\eta-b')
\nearrow (p-p_\eta)+(p_\eta-b')=p-b'.
\eneq
It follows that, for all $\tau\in \Qw,$ 
\beq
d_\tau(e_C)=d_\tau(g(e)-b')\ge \tau(p-b')=\tau(p)-\tau(b')\ge \tau(p)-d_\tau(a)=d_\tau(e)-d_\tau(a).
%\rforal \tau\in \Qw.
\eneq
%
%%%%%%%%%%%%%%%%%%
%
\iffalse
%
%
%%%%%%%%%%%%%%%%%%%%%%
For any  $c\in \Her(e)_+,$ put  
$c'=(p-b')^{1/2}c(p-b_1')^{1/2}.$ 
Then 
\beq
bc'=b(p-b')^{1/2}c(p-b')^{1/2}=0=c'b.
\eneq
In  other words,  $c'\in \Her(b)^\perp\cap \Her(e).$ 

Then, for any $m\in \N$ and any $\tau\in \Qw,$
\beq
\tau(e_C^{1/m})\ge \tau((p-b')^{1/2}g(e))
%\tau(q)\ge \tau((p-b')^{1/2} e^{1/m}(p-b')^{1/2}).
\eneq
Let $m\to\infty,$ we obtain that

Let $q$ be the open projection associated with the hereditary \SCA\, $\Her(b)^\perp\cap \Her(e).$
%Define $e_C=f_{\eta/2}(e)-b'\ge 0.$ 
Then 
\beq
d_\tau(e_C)=d_\tau(f_{\eta/2}(e)-b')\ge \tau(q-
\eneq

Then, for any $m\in \N$ and any $\tau\in \Qw,$
\beq
\tau(q)\ge \tau((p-b')^{1/2} e^{1/m}(p-b')^{1/2}).
\eneq
Let $m\to\infty,$ we obtain that
% for all $\tau\in \Qw,$ 
\beq
\tau(q) &\ge& \tau((p-b')^{1/2}(p-b')^{1/2})\\
&=&\tau(p-b')
\ge d_\tau(e)-d_\tau(a)\hspace{0.4in}\rforal \tau\in \Qw.
%\sum_{i=2}^n h_n(\tau).
%+(\ep/8)\inf\{h_1(\tau): \tau\in \Qw\}.
\eneq
%%%%%%%%%%%%%%%%
\fi
%
%%%%%%%%%%%%%%%%%%%%
Since $e\in {\rm Ped}(A\otimes {\cal K}),$ $d_\tau(e)=\tau(p)<\infty$
for all $\tau\in \wtd{QT}(A)$ (see \cite[(2) of Proposition 2.10]{FLosc}).
Also,
since 
$d_\tau(e)-\tau(q)\ge d_\tau(b)$ for all $\tau\in\Qw,$  we have, by \eqref{Ldig-e-5},
\beq
\tau(q)\le d_\tau(e)-d_\tau(b)
%&=&
%\sum_{n=1}^\infty h_n(\tau)-d_\tau(b_1)\\
&<& d_\tau(e)-d_\tau(a)
%\sum_{n=2}^\infty h_n(\tau)
+\ep/2.
%\inf\{h_1(\tau): \tau\in \Qw\}.
\eneq

\end{proof}

The following is a folklore.  Note that, if $\Delta$ is metrizable, then, for any strictly 
positive lower semicontinuous affine function $g,$ there exists an increasing sequence 
of strictly positive continuous affine functions $\{h_n\}$ such that
$h_n\nearrow g$ pointwisely on $\Delta$ (see Corollary I.1.4 of \cite{Alf}).

\begin{prop}\label{Psincrease}
Let $\Delta$ be a compact convex set and $g$ be a positive lower semicontinuous affine 
function defined on $\Delta$ 
such that there exists a sequence  of  strictly positive continuous affine 
functions $h_n$ such that $h_n\nearrow g$ pointwisely. 
%$f\in Aff_+(\Delta).$ 
Then there exists a sequence  of strictly positive continuous affine functions 
$g_n$
such that
\beq
g=\sum_{n=1}^\infty g_n,
\eneq
where the infinite series converges pointwisely on $\Delta.$
\end{prop}

\begin{proof}
%Recall that 
%we assume that  $QT(A)\not=\emptyset.$ 
%Let $g\in \LAff_+(\overline{QT(A)}^w).$ 
%be such that $\widehat{[a]}=g.$
%Then there is a sequence of increasing $h_n\in \Aff_+(\overline{QT(A)}^w)$ such that
%$h_n\nearrow g$ (pointwisely).

Put $f_1=(1/2)h_1$
%\in {\rm Aff}_+(\Qw)$ 
and $\lambda_1=1/2.$  Then $h_2(\tau)>f_1(\tau)$
for all $\tau\in \Delta.$ Since $f_1$ is continuous and $\Delta$ is compact, one may 
choose $0<\lambda_2<1/2^{2+1}$ such that
$(1-\lambda_2)h_2(\tau)>f_1(\tau)$ for all $\tau\in \Delta$ (see Lemma \ref{Lcomapctcontain}).  Put $f_2=(1-\lambda_2)h_2.$ 
Suppose that $0<\lambda_i<1/2^{2+i}$ is chosen, $i=1,2,...,n,$ such 
that $(1-\lambda_{i+1})h_{i+1}(\tau)> (1-\lambda_i)h_i(\tau)$ for all $\tau\in \Delta,$ $i=1,2,...,n-1.$
Then $h_{n+1}(\tau)>(1-\lambda_n)h_n(\tau)$ for all $\tau\in \Delta.$ Choose 
$0<\lambda_{n+1}<1/2^{2+n}$ such that $(1-\lambda_{n+1})h_{n+1}(\tau)>(1-\lambda_n)h_n(\tau)$ 
for all $\tau\in \Delta.$ Define $f_{n+1}=(1-\lambda_{n+1})h_{n+1}.$
Then $f_{n+1}$ is a strictly positive continuous affine function on $\Delta.$
%\in {\rm Aff}(\Qw)_{++}.$
By induction, we obtain a strictly increasing sequence $\{f_n\}$
of strictly positive continuous affine functions 
%\in {\rm Aff}(\Qw)_{++}$
such that $f_n\nearrow g$ pointwisely on $\Delta.$

Define $g_1=h_1$ and $g_n=f_n-f_{n-1}$ for $n\ge 2.$ 
Then each $g_n$  is a strictly positive continuous affine function on $\Delta$
and $g=\sum_{n=1}^\infty g_n$ (converges pointwisely on $\Delta$). 

%Since $\Gamma$ is surjective, there are $a_n\in (A\otimes {\cal K})_+$
%such that $\widehat{[a_n]}(\tau)=d_\tau(a_n)=h_n,$ $n\in \N.$
%We may assume that $a_ia_j=0$ if $i\not=j.$ 
%Define $b\in (A\otimes {\cal K})_+$ by $b=\sum_{n=1}^\infty a_n/n.$

\end{proof}

\begin{lem}\label{Ldig}
Let $A$ be a $\sigma$-unital algebraically simple \CA\, with strict comparison and 
surjective canonical 
map $\Gamma$ (which is not purely infinite).
%is surjective. 
Suppose that $e\in {\rm Ped}(A\otimes K)_+^{\bf 1}$ is  a nonzero element.
Then, for any $\ep>0,$  there is $e_0\in {\rm Her}(e)_+^{\bf 1}$ 
satisfying the following:

%(1) $\omega(e_0)<\ep$ and 

(1) $d_\tau(e)-d_\tau(e_0)<\ep$ for all $\tau\in \Qw$ and 

(2) there exists an increasing sequence  $\{e_{0,n}\}$ in $\Her(e_0)_+^{\bf 1}$ 
such that $\wh{e_{0,n}}\nearrow \wh{[e_0]}$ and 
$\omega(e_{0,n})<\ep$ for all $n\in \N.$

% $e\in A_+^{\bf 1}$ and

\end{lem}

\begin{proof}
By Lemma \ref{Psincrease},  choose  $h_n\in \Aff_+(\Qw),$ $n\in \N$ such that $h_n(\tau)>0$ for all $\tau\in \Qw$ and 
\beq\label{dig-e-1}
\sum_{n=1}^\infty h_n(\tau)=d_\tau(e)\rforal \tau\in \Qw
\eneq
(the infinite series converges pointwisely).

Since $\Gamma$ is surjective, one may choose  mutually orthogonal 
elements $a_n\in (A\otimes {\cal K})_+^{\bf 1}$ such that
$\wh{[a_n]}=h_n,$ $n\in \Qw.$ Note that $\omega(a_n)=0.$
By Corollary \ref{CPedsen}, $a_n\in {\rm Ped}(A\otimes {\cal K})_+^{\bf 1}.$
%Let $\{e_n\}\in A_+^{\bf 1}$ be an approximate identity for $A.$ 

Note that $h_1(\tau)<d_\tau(e)$ for all $\tau\in \Qw.$ Let $\ep>0.$
By Lemma \ref{Lorth}, 
there exists $b_1\in \Her(e)_+^{\bf 1}$ such that

(i) $h_1(\tau)-\ep
%\inf\{h_1(\tau): \tau\in \Qw\}
<\tau(b_1)\le h_1(\tau)$
for all $\tau\in \Qw$ and 

(ii) $\omega(b_1)<\ep/4,$ 

(iii) $d_\tau(e)-d_\tau(a_1)\le d_\tau(c_1)< (d_\tau(e)-d_\tau(a_1)) +\ep/8$
for all $\tau\in \Qw,$
%(\ep/8)\inf\{h_1(\tau): \tau\in \Qw\},$
where $c_1$ is a positive element of $\Her(b_1)^\perp\cap \Her(e).$
We may assume that $\|c_1\|=1.$

We claim that, 
%for each $n\in \N,$ 
there are mutually orthogonal 
elements $\{b_n\}$ in $\Her(e)_+^{\bf 1}$ such that, for each $m\in \N,$

(i')  $h_m(\tau)-\ep/2^{m+2}
%\inf\{h_m(\tau): \tau\in \Qw\}
<\tau(b_m)\le h_m(\tau)$
for all $\tau\in \Qw,$ 

(ii') $\omega(b_m)<\ep/2^{m+1},$   and

(iii')  for all $\tau\in \Qw,$
 \beq
\hspace{-0.2in}d_\tau(e)-\sum_{i=1}^m d_\tau(a_i) &\le& d_\tau(c_m)
\\
&<&
%< 
(d_\tau(e)-\sum_{i=1}^md_\tau(a_i)) +\sum_{i=1}^m\ep/2^{i+2},
%\inf\{h_n(\tau): \tau\in \Qw\},
\eneq
where $c_m$ is a positive element of $(\Her(\sum_{i=1}^m b_i))^\perp\cap \Her(e)$
with $\|c_m\|=1.$
%$m\in \N.$

Let us assume that the claim holds for $j=1,2,...,m.$
By \eqref{dig-e-1} and   (iii'), we have 
\beq
d_\tau(a_{m+1})=h_{m+1}(\tau)<d_\tau(c_m)\rforal \tau\in \Qw.
\eneq
By Lemma \ref{Lorth},  there is $b_{m+1}\in \Her(c_m)_+^{\bf 1}$ such that

(a) $h_{m+1}(\tau)-\ep/2^{m+3}
%\inf\{h_{m+1}(\tau): \tau\in \Qw\}
<\tau(b_{m+1})\le h_{m+1}(\tau)$
for all $\tau\in \Qw$ and 

(b) $\omega(b_{m+1})<\ep/2^{m+2}$

\noindent 
(so (i') and (ii') holds for $m+1$),

(c) for any $\tau\in \Qw,$ \beq
d_\tau(c_m)-d_\tau(a_{m+1}) &\le& d_\tau(c_{m+1})\\
&<& (d_\tau(c_m)-d_\tau(a_{m+1})) +\ep/2^{m+3},
%\inf\{h_1(\tau): \tau\in \Qw\},
\eneq
where $c_{m+1}$ is a 
%strictly 
positive element of  $\Her(b_{m+1})^\perp\cap \Her(c_m).$
%$\{a\in \Her(c_m): a b_{m+1}=b_{m+1}a=0\}.$ 
Since $c_m\in (\Her(\sum_{i=1}^mb_i))^\perp,$ 
$c_{m+1}\in  (\Her(\sum_{i=1}^{m+1} b_i))^\perp.$ We may assume that $\|c_{m+1}\|=1.$

By (iii') (for $m$) and (c), we estimate that, for all $\tau\in \Qw,$ 
\beq
d_\tau(e)-\sum_{i=1}^{m+1}d_\tau(a_i)&=&(d_\tau(e)-\sum_{i=1}^md_\tau(a_i))-d_\tau(a_{m+1})\\
&\le & d_\tau(c_m)-d_\tau(a_{m+1}) \le d_\tau(c_{m+1})\\
&<& d_\tau(c_m)-d_\tau(a_{m+1}) +\ep/2^{m+3}\\
&<& (d_\tau(e)-\sum_{i=1}^m d_\tau(a_i) +\sum_{i=1}^m \ep/2^{i+2})-d_\tau(a_{m+1})+\ep/2^{m+3}\\
&=& d_\tau(e)-\sum_{i=1}^{m+1} d_\tau(a_i))+\sum_{i=1}^{m+1}\ep/2^{i+2}.
%\inf\{h_n(\tau): \tau\in \Qw\}.
\eneq
By the induction on $m,$  the claim follows.

%For each $b_n,$ by Lemma ???, choose $g_n\in C([0,1])$ with 
%$g(0)=0$ and $0<g(t)\le 1$ for all $t\in [0,\infty)$
%such that $d_n':=g(b_n)\sim b_n$ and 
%\beq
%d_\tau(d_n')-\tau(f_{1/4}(b_n'))<\omega(b_n)+\omega(b_n)/2^{n+2}. 
%\eneq

Define $e_0=\sum_{n=1}^\infty b_n/n.$ Since $b_ib_j=b_j b_i=0,$ if $i\not=j,$ we have
$0\le e_0\le 1$ and $e_0\in \Her(e)_+.$ 
We compute, by (i') and \eqref{dig-e-1},  that
\beq
d_\tau(e_0)=\sum_{n=1}^\infty d_\tau(b_n)\ge \sum_{n=1}^\infty h_n(\tau)-\sum_{n=1}^\infty \ep/2^{m+1}
> d_\tau(e)-\ep\rforal \tau\in \Qw.
\eneq
Thus (1) holds.   To see (2)  holds,  
we first note that, for all $n\in \N,$ since $b_ib_j=b_jb_i=0,$ if $i\not=j,$ and $\|b_n\|\le 1,$ 
and, by (ii')  and Proposition 4.4 of \cite{FLosc}, 
\beq
\omega(\sum_{i=1}^n b_i)\le \sum_{i=1}^n \omega(b_i)<\ep/2\andeqn
\|\sum_{i=1}^n b_i\|\le 1.
\eneq
Define $e_{0,n}=\sum_{i=1}^n b_i^{1/n}=(\sum_{i=1}^n b_i)^{1/n}.$ Then $e_{0, n}\le e_{0,n+1}$ and $e_{0, n}\in \Her(e)_+^{\bf 1},$
$n\in \N.$ 
Moreover,  for any $m\in \N,$  if $n\ge m,$
\beq
\tau(e_{0,n})=\tau((\sum_{i=1}^n b_i)^{1/n})\ge \tau((\sum_{i=1}^m b_i)^{1/n})\rforal \tau\in \Qw.
\eneq
It follows that
\beq
\lim_{n\to\infty}\tau(e_{0,n})\ge d_\tau(\sum_{i=1}^m b_i) \rforal \tau\in \Qw.
\eneq
Let  $m\to\infty.$ We obtain 
\beq
\lim_{n\to\infty}\tau(e_{0,n})\ge d_\tau(e_0)\rforal \tau\in \Qw.
\eneq
As $e_{0,n}\in \Her(e_0)_+$ and $e_{0, n}\le e_{0,n+1}$ for all $n\in \N,$
we conclude that 
\beq
\wh{e_{0,n}}\nearrow  \wh{[e_0]}.
\eneq 
Since $e_{0,n}\sim \sum_{i=1}^n b_i,$  we conclude that  
\beq 
\omega(e_{0,n})=\omega(\sum_{i=1}^n b_i)\le \sum_{i=1}^n \omega(b_i)<\ep/2\rforal n\in \N.
\eneq
%It then follows that 
%Moreover, by (i'),  for all $\tau\in \Qw,$ 
%\beq
%\tau(e_{0,n})=\sum_{i=1}^n \tau(b_i)>\sum_{i=1}^n h_i(\tau)-\sum_{i=1}^n \ep/2^{i+2}
%>\sum_{i=1}^n h_i(\tau)-\ep/2
%\eneq
%for all $n>1.$ 
%\beq
%\omega(e_{0,n})<\ep/2\rforal n\in \N.
%\andeqn
%\tau(e_{0,n})\nearrow d_\tau(e_0).
%\eneq
Hence (2) holds.
%
%%%%%%%%%%%%%%%%%%%%%%%%%%%%%%%
%
%
\iffalse
%%%%%%%%%%%%%%%%%%%%%%%%%%%%%%%%
%
By ?, there exists $x_1\in A_+^{\bf}$ such that
\beq
x_1^*x_1=a_1\andeqn x_1x_1^*\in \Her(e).
\eneq

Since $\wh{a_1}=h_1,$ $\omega(x_1x_1^*)=0.$
Choose $n(1)\in \N$ such that, for all $n\ge n(1),$  
\beq\label{Ldig-e-5}
d_\tau(x_1x_1^*)-\tau(f_{1/n(1)}(x_1x_1^*))<(\ep/8)\inf\{h_1(\tau): \tau\in \Qw\}. 
\eneq
Put $b_1=f_{1/2(n_1)}(x_1x_1^*)$  and $b_1'=f_{1/4n(1)}(x_1x_1^*).$ 
We compute that $\omega(b_1)<(\ep/8)\inf\{h_1(\tau): \tau\in \Qw\}.$ 
Put $c_1=(1-b_1')^{1/2}e(1-b_1')^{1/2}.$ Then 
\beq
b_1e_1=b_1(1-b_1')^{1/2}e(1-b_1')^{1/2}=0=e_1b_1.
\eneq
In  other words,  $c_1\in \Her(b_1)^\perp.$ 
Put $q_1$ be the open projection associated with the hereditary \SCA\, $\Her(b_1)^\perp.$
Then, for any $m\in \N$ and any $\tau\in \Qw,$
\beq
\tau(q_1)\ge \tau((1-b_1')^{1/2} e^{1/m}(1-b')^{1/2}).
\eneq
It follows that, for all $\tau\in \Qw,$ 
\beq
\tau(q_1) &\ge& \tau((1-b_1')^{1/2}(1-b_1')^{1/2})=\tau(1-b_1')\\
&>&\sum_{i=2}^n h_n(\tau).
%+(\ep/8)\inf\{h_1(\tau): \tau\in \Qw\}.
\eneq
Also
since 
$\tau(p)-\tau(q_1)\ge d_\tau(b_1)$ for all $\tau\in\Qw,$  we have, by \eqref{Ldig-e-5},
\beq
\tau(q_1)\le \tau(p)-d_\tau(b_1)&=&\sum_{n=1}^\infty h_n(\tau)-d_\tau(b_1)\\
&<& \sum_{n=2}^\infty h_n(\tau)+(\ep/8)\inf\{h_1(\tau): \tau\in \Qw\}.
\eneq
%%%%%%%%%%
%
\fi
%%%%%%%%%%%%%%%%%%%%%%%%%%%%%%%%
%
%e_1'=(1-b_1')e(1-b_1')\in \overline{(1-b_1)A(1-b_1)}.$
%
\end{proof}

\begin{thm}\label{TT}
Let $A$ be a $\sigma$-unital simple \CA\, with strict comparison and with 
surjective 
canonical map $\Gamma$ (which is not purely infinite).
Then $A$ has tracial approximate oscillation zero.
\end{thm}

\begin{proof}
Choose $e\in {\rm Ped}(A)_+^{\bf 1}\setminus \{0\}.$ Then $\Her(e)$ is algebraically simple.
By Brown's stable isomorphism theorem (\cite{B1}), $A\otimes {\cal K}\cong \Her(e)\otimes {\cal K}.$
Therefore, \wilog, we may assume that $A$ is algebraically simple.

Let $a\in {\rm Ped}(A\otimes {\cal K})_+^{\bf 1}.$ We want to show $\Omega^T(a)=0.$
\Wlog, we may assume that $\|a\|=1.$
It suffices to show that, for any $1>\ep>0,$ there exists $b\in \Her(a)_+^{\bf 1}$
with $\omega(b)<\ep$ and 
\beq
\|a-b\|_{_{2, \Qw}}<\ep.
\eneq

Let  (recall that $0\not\in \Qw,$ see Lemma 4.5 of \cite{eglnp}
and Proposition 2.9 of \cite{FLosc})
\beq
\lambda_a=\inf\{\tau(a): \tau\in \Qw\}>0\andeqn M_a=\sup\{\tau(a):\tau\in \Qw\}<\infty.
\eneq
Set 
\beq
\ep_1=\ep^2/(M_a+1)\andeqn \eta=\ep^2\lambda_a/2(M_a+1)=\ep_1\lambda_a/2.
\eneq
Note that $0<\eta<\ep/2.$

By Lemma \ref{Ldig},  there is $e_0\in {\rm Her}(a)_+^{\bf 1}$ 
satisfying the following:

%(1) $\omega(e_0)<\ep$ and 

(1) $d_\tau(a)-d_\tau(e_0)<\eta$ for all $\tau\in \Qw$ and 

(2) there exists an increasing sequence  $\{e_{0,n}\}$ in $\Her(e_0)_+^{\bf 1}$ 
such that $\wh{e_{0,n}}\nearrow \wh{[e_0]}$ and 
$\omega(e_{0,n})<\eta$ for all $n\in \N.$

Denote by $p$ and $q$ the open projections 
associated with $\Her(a)$ and $\Her(e_0),$ respectively.
Hence (1) is the same as $\tau(p-q)<\eta$ for all $\tau\in \Qw.$

Note that, for all $\tau\in \Qw,$ 
\beq
\tau(a^{1/2}(p-q)a^{1/2})&\le&\|a\|\tau(p-q)<\eta<\ep^2\lambda_a/(M_a+1)\\
&\le &\ep_1 \tau(a)=\tau(a^{1/2}pa^{1/2})-\tau(a^{1/2}(p-\ep_1 p)a^{1/2}).
\eneq
It follows that (recall that $pa=a$) 
\beq
\tau(a^{1/2}qa^{1/2})>\tau(a^{1/2}(p-\ep_1 p)a^{1/2})=(1-\ep_1)\tau(a)\rforal \tau\in \Qw.
\eneq
Note that $(1-\ep_1^2)\wh{a},\, \wh{a^{1/2}e_{0,n}a^{/12}}\in \Aff_+(\Qw)$ and 
$\wh{a^{1/2} e_{0,n} a^{1/2}}\nearrow \wh{a^{1/2}qa^{1/2}}$ point-wisely. It follows 
(see Lemma \ref{Lcomapctcontain}) that, there exists $n_0$ such that, for all $n\ge n_0,$
\beq\label{TT-10}
\tau(a^{1/2} e_{0,n} a^{1/2})>(1-\ep_1)\tau(a)\rforal \tau\in \Qw.
\eneq
Put $b=a^{1/2}e_{0, n_0}a^{1/2}\in \Her(a)_+^{\bf 1}.$ 
Note that $0\le b\le a$ and $0\le a-b\le 1.$ 
Then, by \eqref{TT-10}, 
\beq
\|a-b\|_{_{2, \Qw}}
%a^{1/2}e_{0,n}a^{1/2}\|_{_{2, \Qw}}
&=&\sup\{\tau((a-b)^2)^{1/2}: \tau\in \Qw\}\\
&\le &\sup\{\tau(a-b)^{1/2}: \tau\in \Qw\}\\
&=&\sup\{(\tau(a)-\tau(a^{1/2}e_{0,n_0}a^{1/2}))^{1/2}: \tau\in \Qw\}\\
&<&\sup\{\sqrt{\ep_1\cdot \tau(a)}:\tau\in \Qw\}\\
&=& {\ep\over{\sqrt{M_a+1}}}\cdot \sqrt{\sup\{\tau(a): \tau\in \Qw\}}<\ep.
\eneq
Since $e_{0,n}\in \Her(a)_+,$ $a^{1/2}e_{0, n}a^{1/2}\sim e_{0,n}$ (see Lemma 4.10 of \cite{FLosc}).
It follows that 
\beq
\omega(b)<\ep/2.
\eneq
%=\tau(a)-\eta/2\tau(a)
%
%
Hence $\Omega^T(a)=0.$
\end{proof}

{\bf The proof of Theorem \ref{MT}}:

Recall that,  since $A$ is not purely infinite, by strict comparison, we also assume that $\wtd{QT}(A)\setminus \{0\}\not=\emptyset.$
  By Theorem \ref{TT}, (1)  $\Rightarrow$ (2). 
  That (2) $\Rightarrow$ (1) follows from Theorem 7.11 of \cite{FLosc}.
  
  Suppose (1), or (2) hold. Then, by what has  been proved, both (1) and (2) hold.
  Then, by Theorem 9.4 of \cite{FLosc}, $A$ has stable rank one. 
 This  proves Theorem \ref{MT}.
 
 \vspace{0.2in}

 {\bf The proof of Corollary \ref{CM}}:
 
 Corollary \ref{CM} follows from the combination of Theorem 1.1 of \cite{FLosc} and 
 Theorem \ref{TT}. 
 
 \vspace{0.2in}
 
 By Proposition \ref{CM2}, Corollary \ref{CM2}  follows from the following (see Definition \ref{DpureW}):
 
 \begin{cor}\label{Cpure}
 Let $A$ be a separable  simple pure \CA\, which is not purely infinite. Then 
 $A$ has stable rank one. 
 \end{cor}
 
% {\bf The proof of Corollary \ref{CM2}}
 \begin{proof}
 Since $A$ is not purely infinite, by \cite[Corollary 5.1]{Ror04JS} (see also \cite[Proposition 4.9]{FLsigma}),
 $\wtd{QT}(A)\not=\{0\}.$ 
 By Theorem  \ref{PCuntz},
 %cite{Rosr1}, 
 $A$ has strict comparison and   $\Gamma$ is surjective.  Thus Corollary \ref{CM2} follows 
 from Corollary \ref{CM}. 
 \end{proof}
 
 Return to the question {\bf Q}, we have the following corollary.
 Note that, by, for example, \cite[Proposition 2.1.2]{Rl}, 
 when $A$ has stable rank one, $a\lesssim  b$ implies that $a\simle b.$

 \begin{cor}\label{Ccomparison}
 Let $A$ be a $\sigma$-unital simple \CA\, with $\wtd{QT}(A)\not=\{0\}.$
 Suppose that $A$ has strict comparison and $\Gamma$ is surjective.
 If $a, b\in (A\otimes {\cal K})_+$ with 
 $d_\tau(a)<d_\tau(b)$ for all $\tau\in \wtd{QT}(A)\setminus \{0\},$ 
 then 
 $a\simle b.$
 \end{cor}
 
 Let us return to \ref{NN2} and give a definite answer to the question there.
 Note that the converse of the next corollary also holds as ${\rm Ped}(A\otimes {\cal K})$ is a minimal 
 dense ideal (see \ref{NN2}).
 
 \begin{cor}\label{CPed2}
  Let $A$ be a $\sigma$-unital simple \CA\, with $\wtd{QT}(A)\not=\{0\}.$
 Suppose that $A$ has strict comparison and $\Gamma$ is surjective
 and $S\subset \wtd{QT}(A)\setminus \{0\}$ is a compact subset such that $\R_+\cdot S={\wtd{QT}}(A).$
 If $x\in (A\otimes {\cal K})$ such that $\sup\{d_\tau(x^*x): \tau\in S\}<\infty,$ 
 then $x\in {\rm Ped}(A\otimes {\cal K}).$
 \end{cor}
 
 \begin{proof}
 Put $a=x^*x.$
 Choose $M\in \R_+$ such that
 $M>\sup\{d_\tau(a): \tau\in S\}.$ 
 Define $f\in \Aff_+(S)$ by $f(s)=M$ for all $s\in S$ and $\wtd f\in \Aff_+(\wtd{QT}(A))$ 
 by $\wtd f(r\cdot s)=rf(s)$ for all $s\in S$ and $r\in \R_+.$ 
 Since $\Gamma$ is surjective, we obtain $b\in (A\otimes {\cal K})_+$ such 
 that $\wh{[b]}=\wtd f.$  Since $\wh{[b]}$ is continuous, $\omega(b)=0.$
 It follows from Corollary \ref{CPedsen} that $b\in {\rm Ped}(A\otimes {\cal K})_+.$ 
 
 Now since $d_\tau(a)<d_\tau(b)$ for all $\tau\in {\wtd{QT}}(A)\setminus \{0\},$ 
 by Corollary \ref{Ccomparison},  there exists $y\in A\otimes {\cal K}$ such that
 \beq
 y^*y=a\andeqn yy^*\in \Her(b)\subset {\rm Ped}(A\otimes {\cal K}).
\eneq
It follows that $a\in {\rm Ped}(A\otimes {\cal K})_+,$ whence $x\in  {\rm Ped}(A\otimes {\cal K}).$
 \end{proof}

In the connection to Toms-Winter's Conjecture, with Theorem \ref{MT}, we may restate 
the main result in \cite[Theorem 1.1]{LinOZ} as follows:

\begin{thm}{\rm{(cf. \cite[Theorem 1.1]{LinOZ})}}\label{TM-2}
Let $A$ be a non-elementary separable amenable  simple \CA\, with $\wtd T(A)\setminus \{0\}\not=\emptyset$
such that $\wtd T(A)$ has a
% tight  
$\sigma$-compact countable-dimensional extremal boundary. 
%(see Definition \ref{Dscdim}).
%$\partial_e(T(A))$ 
%strongly countable-dimensional boundary.
%which has a basis $S$ such that  $\partial_e(S)$ has countably many points.
 %such that 
%$\overline{\partial_e(S)}$ has finitely many limit points. 
Then  following are equivalent.

(1) ${\rm Cu}(A)$ is almost unperforated and almost divisible,

(2) $A$ has strict comparison and the canonical map $\Gamma$ is surjective,

(3) $A$ has strict comparison and T-tracial approximate oscillation zero,

(4) $A$ has strict comparison and stable rank one,

(5) $A\cong A\otimes {\cal Z},$
%
%
%(5) 
%
%(4) $A$ has finite nuclear dimension.
\end{thm}

 Note  \cite[Theorem 1.1]{LinOZ} states that, under the same assumption, (3), (4) and (5) are equivalent.
 By  Corollary \ref{CM}, (2) and (3) are equivalent (see Definition \ref{DOT}).  By Theorem  \ref{PCuntz}, 
 (1) and (2) are equivalent. 
 %Theorem \ref{MT}, 
 So we  are able to add (1) and (2) to the statement of Theorem \ref{TM-2}.
 
 \vspace{0.2in}
 
 Let us end this paper with the following question:
 Let $A$ be a separable simple \CA\, with strict comparison and surjective $\Gamma.$
 Is $A$ tracially approximately  divisible (see \cite{FLL})? If, in addition, $A$ is  assumed to be amenable, is 
 $A$ ${\cal Z}$-stable?

\vspace{0.4in}

\noindent email: hlin@uoregon.edu
\end{document}